\documentclass[a4paper,12pt]{article}
\usepackage{tikz}
\usetikzlibrary{matrix,arrows,decorations.markings}
\usepackage{amsmath,amsfonts,amssymb,amsthm,url,textcomp}
\usepackage{csquotes}
\usepackage{a4wide}
\usepackage[numbers]{natbib}
\usepackage{fancyhdr}
\usepackage{anyfontsize}
\usepackage{hyperref}
\usepackage{hhline}
\usepackage{leftidx}
\usepackage{enumitem}
\usepackage{mdframed}
\usepackage[ruled,vlined]{algorithm2e}
\usepackage{makecell}

\allowdisplaybreaks

\makeatletter
\let\@fnsymbol\@arabic
\makeatother

\newtheorem{theorem}{Theorem}\numberwithin{theorem}{section}
\newtheorem{definition}[theorem]{Definition}
\newtheorem*{definition*}{Definition}
\newtheorem*{proposition*}{Proposition}
\newtheorem{lemma}[theorem]{Lemma}
\newtheorem{corollary}[theorem]{Corollary}
\newtheorem{proposition}[theorem]{Proposition}
\newtheorem{notation}[theorem]{Notation}

\newtheorem{question}[theorem]{Question}

\newtheorem{theoremm}{Theorem}\numberwithin{theoremm}{subsection}

\numberwithin{theoremmm}{subsubsection}

\theoremstyle{remark}

\newtheorem{example}[theorem]{Example}
\newtheorem*{example*}{Example}

\newcommand{\lcm}{\operatorname{lcm}}

\newcommand{\ord}{\operatorname{ord}}
\newcommand{\Sym}{\operatorname{Sym}}

\newcommand{\id}{\operatorname{id}}

\newcommand{\Mod}[1]{\ (\textup{mod}\ #1)}

\newcommand{\IN}{\mathbb{N}}

\newcommand{\IF}{\mathbb{F}}

\newcommand{\IZ}{\mathbb{Z}}

\newcommand{\CT}{\operatorname{CT}}
\newcommand{\supp}{\operatorname{supp}}

\newcommand{\Mcal}{\mathcal{M}}

\newcommand{\Ccal}{\mathcal{C}}
\newcommand{\codom}{\operatorname{codom}}

\newcommand{\ran}{\operatorname{ran}}

\newcommand{\init}{\mathrm{init}}
\newcommand{\term}{\mathrm{term}}
\newcommand{\IC}{\mathbb{C}}
\newcommand{\Null}{\operatorname{Null}}
\newcommand{\IQ}{\mathbb{Q}}
\newcommand{\reg}{\mathrm{reg}}

\begin{document}

\title{Cycle types of complete mappings of finite fields}

\author{Alexander Bors\textsuperscript{1} \and Qiang Wang\thanks{School of Mathematics and Statistics, Carleton University, 1125 Colonel By Drive, Ottawa ON K1S 5B6, Canada. \newline First author's e-mail: \href{mailto:alexanderbors@cunet.carleton.ca}{alexanderbors@cunet.carleton.ca} \newline Second author's e-mail: \href{mailto:wang@math.carleton.ca}{wang@math.carleton.ca} \newline The authors are supported by the Natural Sciences and Engineering Research Council of Canada (RGPIN-2017-06410). \newline 2020 \emph{Mathematics Subject Classification}: Primary: 11T22. Secondary: 05A05, 11T24, 11T55. \newline \emph{Keywords and phrases}: Finite fields, Cyclotomy, Cyclotomic mappings, Complete mappings, Cycle structure.}}

\date{\today}

\maketitle

\abstract{We derive several existence results concerning cycle types and, more generally, the \enquote{mapping behavior} of complete mappings. Our focus is on so-called first-order cyclotomic mappings, which are functions on a finite field $\IF_q$ that fix $0$ and restrict to the multiplication $x\mapsto a_ix$ by a fixed element $a_i\in\IF_q$ on each coset $C_i$ of a given subgroup $C$ of $\IF_q^{\ast}$. The gist of two of our main results is that as long as $q$ is large enough relative to the index $|\IF_q^{\ast}:C|$, all cycle types of first-order cyclotomic permutations with only long cycles on $\IF_q^{\ast}$ can be achieved through a complete mapping, as can all permutations of the cosets of $C$. Our third main result provides new examples of complete mappings $f$ such that both $f$ and its associated orthomorphism $f+\id$ permute the nonzero field elements in one cycle.}

\section{Introduction}\label{sec1}

\subsection{Background and main results}\label{subsec1P1}

Let $G$ be an additive (though not necessarily abelian) group. A function $f:G\rightarrow G$ is called a \emph{complete mapping of $G$} if both $f$ and $f+\id:g\mapsto f(g)+g$ are permutations of $G$. In case $G$ is the underlying additive group of a field $K$, a complete mapping of $G$ is also called a \emph{complete mapping of $K$}.

Complete mappings were originally introduced by Mann in 1942 as a means to construct mutually orthogonal Latin squares \cite{Man42a}. The existence problem for complete mappings sparked interest in the group-theoretic community early on, with Bateman showing in 1950 that every infinite group admits a complete mapping \cite{Bat50a}, and Hall and Paige showing in 1955 that that a finite group does \emph{not} admit complete mappings if its Sylow $2$-subgroup is cyclic and nontrivial \cite{HP55a}. The converse of Hall and Paige's result is the celebrated Hall-Paige Conjecture, which is now a theorem thanks to the work of Wilcox \cite{Wil09a}, Evans \cite{Eva09a} and Bray \cite[Section 2]{BCCSZ20a}.

Complete mappings of finite fields have also been extensively studied, especially with regard to their polynomial representations. Niederreiter and Robinson's 1984 paper \cite{NR84a} is an early (and influential) contribution in this context. Later, practical applications of complete mappings such as check-digit systems \cite{Sch00a,SW10a} and the construction of cryptographic functions \cite{MP14a,SGCGM12a} were discovered, and this spurred great interest in them, see e.g.~\cite{ITW17a,TZH14a,Win14a,WLHZ14a,XC15a,ZHC15a}.

Recall that intuitively speaking, the \emph{cycle type} of a permutation of a finite set is the information how many (disjoint) cycles of each given length the permutation has (see also the beginning of Section \ref{sec2} for a formal definition). This paper is a contribution to the study of cycle types of complete mappings of finite fields (or, equivalently, of finite elementary abelian groups). Despite the investigations of a few authors such as \cite{FGT81a} or \cite[Section 4]{MP14a} (and \cite{SC16a} for the related concept of \emph{orthomorphisms}, permutations $f:G\rightarrow G$ such that $f-\id$ is also a permutation), our understanding of this aspect of complete mappings appears to be very limited still. For example, there are hardly any methods to identify or refute a given cycle type of a permutation of $q$ elements as the cycle type of some complete mapping of $\IF_q$.

The results of this paper are positive in the sense that they provide examples of cycle types or, more generally, \enquote{mapping behaviors} that can be achieved, rather than ways to refute a given cycle type as not possible for a complete mapping. The functions $\IF_q\rightarrow\IF_q$ that we will work with are of a particular form and are known as \enquote{first-order cyclotomic mappings}. More precisely, for a divisor $d$ of $q-1$, an \emph{index $d$ first-order cyclotomic mapping of $\IF_q$} is a function $f:\IF_q\rightarrow\IF_q$ such that
\[
f(x)=\begin{cases}0, & \text{if }x=0, \\ a_ix, & \text{if }x\in C_i\text{ for some }i=0,1,\ldots,d-1,\end{cases}
\]
where $a_0,a_1,\ldots,a_{d-1}\in\IF_q$ are fixed and $C_i=\omega^iC$ where $\omega$ is a fixed primitive root of $\IF_q$ and $C$ denotes the index $d$ subgroup of $\IF_q^{\ast}$. That is, on each coset $C_i$ of $C$ in $\IF_q^{\ast}$, $f$ is the multiplication by a fixed element $a_i$ of $\IF_q$. In particular, the index $1$ first-order cyclotomic mappings of $\IF_q$ are just the functions $x\mapsto ax$ for all $x\in\IF_q$, and first-order cyclotomic mappings of larger indices are natural generalizations thereof. Various authors have studied cyclotomic mappings and generalizations, see e.g.~\cite{Bel13a,NW05a,Wan07a,Wan13a,ZYZP16a}.

Note that by definition, every index $d$ first-order cyclotomic \emph{permutation} of $\IF_q$ (i.e., index $d$ first-order cyclotomic mapping of $\IF_q$ that is a permutation of $\IF_q$) has $0$ as a fixed point. Our following first main result states that as long as the cycle type of an index $d$ first-order cyclotomic permutation of $\IF_q$ has only \enquote{long} cycles apart from the one obligatory fixed point, it can be achieved as the cycle type of an index $d$ first-order cyclotomic complete mapping of $\IF_q$, provided that $q$ is large enough. In fact, Theorem \ref{mainTheo1} is even more general than that and allows one to demand that not only $f+\id$, but a fixed finite number $f+c_1\id,f+c_2\id,\ldots,f+c_n\id$ of \enquote{translates} of $f$ be permutations -- so, for example, it also covers the cases of orthomorphisms and strong complete mappings (functions that are both orthomorphisms and complete mappings).

\begin{theoremm}\label{mainTheo1}
Let $d$ and $n$ be positive integers, and let $\epsilon>0$. There is a $q_0=q_0(d,n,\epsilon)>0$ such that for all prime powers $q\geq q_0$ and all $c_1,c_2,\ldots,c_n\in\IF_q$, every cycle type of a first-order cyclotomic permutation of $\IF_q$ of index $d$ all of whose cycles on $\IF_q^{\ast}$ are of length at least $\epsilon q$ can be realized as the cycle type of a first-order cyclotomic permutation $f$ of $\IF_q$ such that $f+c_1\id,f+c_2\id,\ldots,f+c_n\id$ are permutations of $\IF_q$.
\end{theoremm}

For example, if $q\equiv1\Mod{36}$, then $\IF_q$ has a first-order cyclotomic permutation of index $3$ with one fixed point (the field element $0$), three cycles of length $\frac{2}{9}(q-1)$ and four cycles of length $\frac{1}{12}(q-1)$ (see Example \ref{psiFunEx}). Theorem \ref{mainTheo1} guarantees that as long as $q$ is large enough, there is actually an index $3$ first-order cyclotomic (strong) complete mapping of $\IF_q$ with that cycle structure (see also Example \ref{psiFunEx2}).

We will be able to prove Theorem \ref{mainTheo1} using a theorem of Carlitz, \cite[Theorem 1]{Car56a}, which we will recall as Theorem \ref{carlitzTheo} in Section \ref{sec3}. This theorem grants some control over the multiplicative orders of several polynomial evaluations $Q_1(\xi),Q_2(\xi),\ldots,Q_r(\xi)$ at a single element $\xi\in\IF_q$.

Our next main result is not concerned with cycle types \emph{per se}, but rather with the \enquote{rough mapping behavior} of index $d$ cyclotomic mappings of $\IF_q$ (i.e., their mapping behavior on cosets of the index $d$ subgroup of $\IF_q^{\ast}$). It should not be surprising that this is easier to control than the precise mapping behavior (on elements of $\IF_q$) of cyclotomic mappings, and in fact, we have the following counterpart to Theorem \ref{mainTheo1}, which shows that complete control can be achieved:

\begin{theoremm}\label{mainTheo2}
Let $d$ and $n$ be positive integers. There is a $q_1=q_1(d,n)>0$ such that for all prime powers $q\geq q_1$ with $q\equiv1\Mod{d}$, the following holds: Let $c_1,c_2,\ldots,c_n\in\IF_q$ be pairwise distinct, and let $s_1,s_2\ldots,s_n$ be functions $M\rightarrow M$ where $M=\IF_q^{\ast}/C$ is the set of all cosets in $\IF_q^{\ast}$ of the index $d$ subgroup $C$ of $\IF_q^{\ast}$. Then there is a first-order cyclotomic mapping $f$ of $\IF_q$ of index $d$ such that for all $j=1,2,\ldots,n$ and all cosets $\Ccal$ of $C$ in $\IF_q^{\ast}$, one has $(f+c_j\id)(\Ccal)=s_j(\Ccal)$.
\end{theoremm}

We note that the proof of Theorem \ref{mainTheo2} will require a stronger version of Carlitz's original theorem, stated as Theorem \ref{carlitzGenTheo} in this paper. This theorem provides not only control over the multiplicative orders of $Q_1(\xi),\ldots,Q_r(\xi)$, but also over the integer remainders upon divison by $d$ of their discrete logarithms $\log_{\omega}(Q_1(\xi)),\ldots,\log_{\omega}(Q_r(\xi))$ with regard to a fixed primitive root $\omega$ of $\IF_q$.

Finally, we consider the problem of simultaneously controlling the cycle structures of a complete mapping $f$ of $\IF_q$ and of its associated orthomorphism $f+\id$, which appears to be much harder than merely controlling their mapping behavior on cosets (as in Theorem \ref{mainTheo2}). One of the many interesting consequences of Carlitz's theorem mentioned above is that if $q$ is a large enough prime power, there is a primitive root $\omega$ of $\IF_q$ such that $\omega+1$ is also a primitive root of $\IF_q$. Consequently, the multiplication by $\omega$ on $\IF_q$ is an additive complete mapping $f:\IF_q\rightarrow\IF_q$ such that both $f$ and its associated orthomorphism $f+\id$ permute the elements of $\IF_q^{\ast}$ cyclically. We will explore these ideas further and adapt Carlitz's proof to obtain the following result, which shows that for many $q$, there are more examples of such complete mappings of $\IF_q$ than just additive ones:

\begin{theoremm}\label{mainTheo3}
Let $d>1$ be an integer. There is a constant $C=C(d)>0$ such that if $q$ is a prime power with $q\equiv1\Mod{d}$ and $q\geq C$, then the finite field $\IF_q$ admits a complete mapping $f$ with $f(0)=0$ such that both $f$ and $f+\id$ permute $\IF_q^{\ast}$ cyclically and $f$ is not an additive function $\IF_q\rightarrow\IF_q$.
\end{theoremm}

Note that Theorem \ref{mainTheo3} is concerned with the case where both $f$ and $f+\id$ have a $(q-1)$-cycle (and one fixed point). We remark that this is the \enquote{best situation} as far as simultaneous long cycles of $f$ and $f+\id$ are concerned when $q$ is even, and it is \enquote{almost the best situation} in general -- for more details, see Section \ref{sec7}, which discusses some open problems relating to Theorem \ref{mainTheo3}.

\subsection{Overview of this paper}\label{subsec1P2}

In Section \ref{sec2}, we formally introduce the concept of the cycle type $\CT(\sigma)$ of a permutation $\sigma$ of a finite set. We then proceed to analyze the possible cycle types of index $d$ first-order cyclotomic permutations of $\IF_q$, see Lemma \ref{cyclotomicCTLem}. This will allow us to conveniently enumerate these cycle types using the auxiliary concept of a \emph{$\Sym(d)$-function}, see Definition \ref{psiFunDef} and Lemma \ref{psiFunLem}. These results are preparations for the later proof of Theorem \ref{mainTheo1}.

In Section \ref{sec3}, we review and strengthen Carlitz's theorem \cite[Theorem 1]{Car56a} on the multiplicative orders of polynomial evaluations $Q_1(\xi),\ldots,Q_r(\xi)$. More precisely, Carlitz's theorem provides an asymptotic counting formula for the elements $\xi\in\IF_q$ such that $\ord(Q_i(\xi))=e_i$ for $i=1,2,\ldots,r$, where $e_1,\ldots,e_r$ are fixed divisors of $q-1$ (and $Q_1,\ldots,Q_r$ are non-constant, squarefree and pairwise coprime univariate polynomials over $\IF_q$). This counting is performed by
\begin{itemize}
\item expressing the characteristic function of the set of order $e$ elements of $\IF_q^{\ast}$ as a character sum, see Lemma \ref{carlitzLem}, and
\item applying Lemma \ref{davenportLem}, a fundamental lemma of Davenport, \cite[Formula (6) on p.~100]{Dav39a}, in a quantitatively strengthened form by virtue of a result of Weil, \cite{Wei41a}.
\end{itemize}
For proving our strengthened version of Carlitz's theorem, Theorem \ref{carlitzGenTheo}, we will follow the same basic approach, using that the characteristic function of the set of nonzero $\xi\in\IF_q$ with $\log_{\omega}(\xi)\equiv i\Mod{d}$ (for fixed integers $i$ and $d$) also has a simple representation as a character sum, see Lemma \ref{discLogLem}.

Each of Sections \ref{sec4}--\ref{sec6} is concerned with the proof of one of our main results:
\begin{itemize}
\item In Section \ref{sec4}, we prove Theorem \ref{mainTheo1}. The proof idea is to construct, for a given cycle type as in the theorem, an index $d$ first-order cyclotomic permutation $f_{\omega}$ of $\IF_q$ of a particular form, which has the given cycle type and depends on a primitive root $\omega$ of $\IF_q$. In order to show that the translates $f_{\omega}+c_j\id$ for $j=1,2,\ldots,n$ of $f_{\omega}$ are also permutations of $\IF_q$, we exploit that they are also index $d$ first-order cyclotomic mappings of $\IF_q$, and the coset-wise constants by which these functions multiply are certain polynomial evaluations at $\omega$. This allows us to apply Carlitz's original theorem to force these constants to lie in $C$, thus proving the theorem.
\item Section \ref{sec5} is dedicated to the proof of Theorem \ref{mainTheo2}. It turns out that this is a relatively straightforward application of our strengthened version of Carlitz's theorem, Theorem \ref{carlitzGenTheo}. Since we are only interested in the way the $d$ cosets $\omega^iC$, for $i=0,1,\ldots,d-1$, are mapped to each other, we care about the remainders modulo $d$ of the discrete logarithms with regard to $\omega$ of the coset-wise multiplication constants of the translates $f+c_j\id$, and these remainders can be controlled with Theorem \ref{carlitzGenTheo}.
\item In Section \ref{sec6}, we prove Theorem \ref{mainTheo3}. This, too, is done by considering index $d$ first-order cyclotomic permutations of $\IF_q$ of a specific form, but the problem is that the polynomials whose evaluations one wants to control are not pairwise coprime anymore. Hence one cannot apply Theorem \ref{carlitzGenTheo} here, but rather, one must analyze the counting formula using character sums separately and hope to be able to show that the main term still dominates over the error term, which turns out to be the case.
\end{itemize}

Finally, in Section \ref{sec7}, we discuss three open problems that are related to Theorem \ref{mainTheo3} and may serve as inspiration for further research.

\subsection{Notation and terminology}\label{subsec1P3}

We denote by $\IN$ the set of natural numbers (including $0$) and by $\IN^+$ the set of positive integers. Throughout this paper, the symbol $\phi$ denotes Euler's totient function, and $\mu$ denotes the M{\"o}bius function. The codomain and range of a function $f$ are denoted by $\codom(f)$ and $\ran(f)$ respectively.

For $d\in\IN^+$, the symmetric group of degree $d$, denoted $\Sym(d)$, is normally defined as the group of all permutations of the set $\{1,2,\ldots,d\}$, but for the purposes of this paper, it will be more convenient to define it as the group of all permutations of the set $\{0,1,\ldots,d-1\}$, so we do that.

If $q$ is a prime power, and $\IF_q$ is the finite field of size $q$, then for each primitive root $\omega$ of $\IF_q$ and each nonzero field element $\xi\in\IF_q^{\ast}$, we denote by $\log_{\omega}(\xi)$ the \emph{discrete logarithm of $\xi$ with respect to $\omega$} -- this is the unique integer $e\in\{0,1,\ldots,q-2\}$ such that $\xi=\omega^e$.

\section{Cycle types of first-order cyclotomic permutations}\label{sec2}

In this section, we briefly review the general concept of a cycle type as well as how to determine the cycle type of a first-order cyclotomic permutation of a finite field.

Let $\psi$ be a permutation of a finite set $X$ with $|X|=d$. As noted in the introduction, intuitively, the \emph{cycle type of $\psi$} is the information how many cycles of each length $\psi$ has. Formally, the cycle type of $\psi$ is defined as the monomial
\[
\CT(\psi):=x_1^{k_1}x_2^{k_2}\cdots x_d^{k_d}\in\IQ[x_1,\ldots,x_d]
\]
where $k_{\ell}$ denotes the number of length $\ell$ cycles of $\psi$ for $\ell=1,2,\ldots,d$. For example, if $\psi=(0,1,2)(3,4)(5,6)(7)\in\Sym(8)$, then
\[
\CT(\psi)=x_1x_2^2x_3.
\]
Lemma \ref{cyclotomicCTLem} below describes how to compute the cycle types of cyclotomic permutations of finite fields. For its formulation, we need some notation. Observe that when $f$ is an index $d$ cyclotomic permutation of $\IF_q$, then $f$ permutes the cosets of the index $d$ subgroup $C$ of $\IF_q^{\ast}$. Hence, if $\omega$ is a primitive root of $\IF_q$, and if we enumerate the cosets of $C$ according to $\omega$ as $C_i:=\omega^iC$ for $i=0,1,\ldots,d-1$, then there is a unique permutation $\psi=\psi_{f,\omega}\in\Sym(d)$ such that $f(C_i)=C_{\psi(i)}$.

\begin{notation}\label{piNot}
Let $d$ be a positive integer, let $q$ be a prime power with $q\equiv1\Mod{d}$, let $\omega$ be a primitive root of $\IF_q$, and let $C$ be the index $d$ subgroup of $\IF_q^{\ast}$. Moreover, let $f$ be an index $d$ first-order cyclotomic permutation of $\IF_q$, of the form
\[
f(x)=\begin{cases}0, & \text{if }x=0, \\ a_ix, & \text{if }x\in C_i=\omega^iC\text{ for some }i=0,1,\ldots,d-1.\end{cases}
\]
Let $\psi$ be the unique permutation in $\Sym(d)$ such that $f(C_i)=C_{\psi(i)}$ for $i=0,1,\ldots,d-1$. For each cycle $\zeta$ of $\psi$, we set
\[
\pi(\zeta)=\pi_{f,\omega}(\zeta):=\prod_{i\in\supp(\zeta)}{a_i}\in C.
\]
Moreover, we denote by $\ell(\zeta)$ the length of $\zeta$ and by $\supp(\zeta)$ the \emph{support of $\zeta$}, which is the set $\{i_1,i_2,\ldots,i_{\ell(\zeta)}\}$ of all indices in $\{0,1,\ldots,d-1\}$ that lie on $\zeta=(i_1,i_2,\ldots,i_{\ell(\zeta)})$ (note that $\supp(\zeta)=\{i\}$ in case $\zeta=(i)$ is a cycle of length $1$).
\end{notation}

\begin{lemma}\label{cyclotomicCTLem}
With the same assumptions as in Notation \ref{piNot}, we have that
\[
\CT(f)=x_1\prod_{\zeta}{x_{\ell(\zeta)\ord(\pi(\zeta))}^{(q-1)/(d\ord(\pi(\zeta)))}}
\]
where the running index $\zeta$ ranges over the (disjoint) cycles of $\psi$.
\end{lemma}

\begin{proof}
If $x\in C_i$ for some $i\in\{0,1,\ldots,d-1\}$, then for each $\ell\in\IN^+$, we have that $f^{\ell}(x)\in C_{\psi^{\ell}(i)}$. In particular, a necessary condition for $f^{\ell}(x)=x$ is that $\psi^{\ell}(i)=i$, which shows that the cycle length of $x$ under $f$ is divisible by the length $\ell(\zeta(i))$ of the unique cycle $\zeta(i)$ of $\psi$ that contains $i$. Equivalently, we have that the cycle length of $x$ under $f$ is the $\ell(\zeta(i))$-fold of the cycle length of $x$ under the iterate $f^{\ell(\zeta(i))}$.

Now, iterating $f$ on $x$ for $\ell(\zeta(i))$ times, we find that
\begin{align*}
&(x\in C_i) \mapsto (a_ix\in C_{\psi(i)}) \mapsto (a_ia_{\psi(i)}x\in C_{\psi^2(i)}) \mapsto \cdots \mapsto \\
&(a_ia_{\psi(i)}\cdots a_{\psi^{\ell(\zeta(i))-1}(i)}x=\pi(\zeta(i))x\in C_{\psi^{\ell(\zeta(i))}(i)}=C_i).
\end{align*}
Since $x$ and $\pi(\zeta(i))x$ lie in the same coset $C_i$ of the multiplicative subgroup $C$, we conclude that $\pi(\zeta(i))\in C$. Moreover, every cycle of the multiplication by $\pi(\zeta(i))$ on $C_i$ is of length $\ord(\pi(\zeta(i)))$ -- in particular, the cycle length of $x$ under $f^{\ell(\zeta(i))}$ must be equal to $\ord(\pi(\zeta(i)))$, whence the cycle length of $x$ under $f$ equals $\ell(\zeta(i))\cdot\ord(\pi(\zeta(i)))$.

In summary, this shows that for each cycle $\zeta$ of $\psi$, on the union $U(\zeta):=\bigcup_{i\in\supp(\zeta)}{C_i}$, the permutation $f$ decomposes into
\[
\frac{|\bigcup_{i\in\supp(\zeta)}{C_i}|}{\ell(\zeta)\cdot\ord(\pi(\zeta))}=\frac{\ell(\zeta)(q-1)/d}{\ell(\zeta)\cdot\ord(\pi(\zeta))}=\frac{q-1}{d\ord(\pi(\zeta))}
\]
disjoint cycles, each of length $\ell(\zeta)\cdot\ord(\pi(\zeta))$. In other words, we have that
\[
\CT(f_{\mid U(\zeta)})=x_{\ell(\zeta)\ord(\pi(\zeta))}^{(q-1)/(d\ord(\pi(\zeta)))}.
\]
But the sets $U(\zeta)$ for the various cycles $\zeta$ of $\psi$ form a partition of $\IF_q^{\ast}$. This implies that
\[
\CT(f_{\mid{\IF_q^{\ast}}})=\prod_{\zeta}{\CT(f_{\mid U(\zeta)})}=\prod_{\zeta}{x_{\ell(\zeta)\ord(\pi(\zeta))}^{(q-1)/(d\ord(\pi(\zeta)))}}.
\]
Similarly, we have
\[
\CT(f)=\CT(f_{\mid\{0\}})\cdot\CT(f_{\mid{\IF_q^{\ast}}})=x_1\cdot\prod_{\zeta}{\CT(f_{\mid U(\zeta)})}=x_1\prod_{\zeta}{x_{\ell(\zeta)\ord(\pi(\zeta))}^{(q-1)/(d\ord(\pi(\zeta)))}},
\]
as required.
\end{proof}

We remark that Lemma \ref{cyclotomicCTLem} could also be proved as a special case of a more general theory (which we will not need in this paper). Indeed, the restrictions to $\IF_q^{\ast}$ of index $d$ first-order cyclotomic permutations of $\IF_q$ form a permutation group on $\IF_q^{\ast}$ that is isomorphic (as a permutation group) to the imprimitive permutational wreath product $C_{\reg}\wr\Sym(d)$, where $C_{\reg}$ denotes the (right-)regular permutation representation of the abstract group $C$ on itself (see \cite[Theorem 1.1.3(2) and its proof]{BW21a} for more details). The study of cycle types of wreath product elements dates back to a 1937 paper of P{\'o}lya, \cite{Pol37a}. In fact, in \cite[table at the bottom of p.~180]{Pol37a}, P{\'o}lya gives a concise formula for the so-called \emph{cycle index} of an arbitrary imprimitive permutational wreath product, which encodes the \enquote{global} information how many elements of each given cycle type there are in the wreath product. P{\'o}lya's result has a \enquote{local} counterpart (see \cite[Lemma 3.5]{BW21a}), which describes the cycle type of a given wreath product element, and Lemma \ref{cyclotomicCTLem} may be viewed as a special case of that.

Based on Lemma \ref{cyclotomicCTLem}, we now introduce some technical concepts that we will need for the proof of Theorem \ref{mainTheo1}:

\begin{definition}\label{psiFunDef}
Let $d$ be a positive integer, and let $\psi\in\Sym(\{0,1,\ldots,d-1\})$.
\begin{enumerate}
\item A \emph{$\psi$-function} is a function from the set of cycles of $\psi$ to $\IN^+$.
\item A \emph{$\Sym(d)$-function} is a function that is a $\sigma$-function for some $\sigma\in\Sym(d)$.
\item Let $h$ be a $\psi$-function, and recall that $\ran(h)$ denotes the \emph{range of $h$} (the set of all function values of $h$).
\begin{enumerate}
\item We say that a prime power $q$ is \emph{$h$-admissible} (or, equivalently, that $h$ is \emph{$q$-admissible}) if $d\cdot\lcm{\ran(h)}$ divides $q-1$.
\item We associate with $h$ a function $\gamma_h$ from the set of all $h$-admissible prime powers to the set of all monomials in $\IQ[x_n: n\in\IN^+]$ such that for each $h$-admissible prime power $q$, one has
\[
\gamma_h(q):=x_1\prod_{\text{cycles }\zeta\text{ of }\psi}{x_{\ell(\zeta)(q-1)/(d\cdot h(\zeta))}^{h(\zeta)}}.
\]
\end{enumerate}
\item Let $q$ be a prime power with $q\equiv1\Mod{d}$, let $\omega$ be a primitive root of $\IF_q$, let $C$ be the index $d$ subgroup of $\IF_q^{\ast}$, and denote by $C_i:=\omega^iC$ for $i=0,1,\ldots,d-1$ the various cosets of $C$ in $\IF_q^{\ast}$. Moreover, let $f$ be an index $d$ first-order cyclotomic permutation of $\IF_q^{\ast}$ such that $f(C_i)=C_{\psi(i)}$ for $i=0,1,\ldots,d-1$. The \emph{$\Sym(d)$-function of $f$ (with regard to $\omega$)}, denoted by $h_{f,\omega}$, is the $\psi$-function that maps each cycle $\zeta$ of $\psi$ to
\[
\frac{q-1}{d\ord(\pi(\zeta))}
\]
where $\pi(\zeta)=\pi_{f,\omega}(\zeta)$ is as in Notation \ref{piNot}.
\end{enumerate}
\end{definition}

The following simple lemma characterizes cycle types of index $d$ first-order cyclotomic permutations of a finite field in terms of $\Sym(d)$-functions:

\begin{lemma}\label{psiFunLem}
Let $d$ be a positive integer. The following hold:
\begin{enumerate}
\item Let $q$ be a prime power, let $\omega$ be a primitive root of $\IF_q$, and let $f$ be an index $d$ first-order cyclotomic permutation of $\IF_q$. Then $q$ is an $h_{f,\omega}$-admissible prime power, and $\CT(f)=\gamma_{h_{f,\omega}}(q)$.
\item Conversely, assume that $h$ is a $\Sym(d)$-function, and that $q$ is an $h$-admissible prime power. Then for each primitive root $\omega$ of $\IF_q$, there is an index $d$ first-order cyclotomic permutation $f$ of $\IF_q$ such that $h_{f,\omega}=h$ and, consequently, $\gamma_h(q)=\CT(f)$.
\end{enumerate}
In particular, for each prime power $q$ with $q\equiv1\Mod{d}$, the cycle types of index $d$ first-order cyclotomic permutations of $\IF_q$ are precisely the monomials of the form $\gamma_h(q)$ where $h$ is a $q$-admissible $\Sym(d)$-function.
\end{lemma}

\begin{proof}
Statement (1) is immediate from Lemma \ref{cyclotomicCTLem} and Definition \ref{psiFunDef}(4). For statement (2), let $\psi$ be the unique permutation in $\Sym(d)$ such that $h$ is a $\psi$-function, and for each cycle $\zeta$ of $\psi$, fix an element $b_{\zeta}\in\IF_q^{\ast}$ such that
\[
\frac{q-1}{d\ord(b_{\zeta})}=h(\zeta).
\]
In particular, $\ord(b_{\zeta})$ divides $\frac{q-1}{d}$, whence $b_{\zeta}\in C$. Denoting by $\zeta(i)$ for $i=0,1,\ldots,d-1$ the unique cycle of $\psi$ such that $i\in\supp(\zeta(i))$, set
\[
f(x):=\begin{cases}0, & \text{if }x=0, \\ \omega^{\psi(i)-i}x, & \text{if }x\in C_i\text{ and }i\not=\max(\supp(\zeta(i))), \\ \omega^{\psi(i)-i}b_{\zeta(i)}x, & \text{if }x\in C_i\text{ and }i=\max(\supp(\zeta(i))).\end{cases}
\]
Then $f:\IF_q\rightarrow\IF_q$ is clearly an index $d$ cyclotomic mapping of $\IF_q$, and since $b_{\zeta(i)}\in C$, we see by definition of $f$ that for each $i=0,1,\ldots,d-1$, we have $f(C_i)=\omega^{\psi(i)-i}C_i=C_{\psi(i)}$. Hence $f$ is an index $d$ cyclotomic permutation of $\IF_q$ such that $h_{f,\omega}$ is a $\psi$-function. Moreover, for each (length $\ell$) cycle $\zeta=(i_1,i_2,\ldots,i_{\ell})$ of $\psi$, say written such that $i_{\ell}=\max(\supp(\zeta))$, we find that
\[
\pi_{f,\omega}(\zeta)=\omega^{i_2-i_1}\omega^{i_3-i_2}\cdots\omega^{i_{\ell}-i_{\ell-1}}\cdot\omega^{i_1-i_{\ell}}b_{\zeta}=b_{\zeta},
\]
whence
\[
h_{f,\omega}(\zeta)=\frac{q-1}{d\ord(b_{\zeta})}=h(\zeta),
\]
as required.
\end{proof}

\begin{example}\label{psiFunEx}
Let $d=3$, let $\psi=(0,1)(2)\in\Sym(3)$, and let $h$ be the $\psi$-function with $h((0,1))=3$ and $h((2))=4$. Then
\[
d\cdot\lcm(\ran(h))=3\cdot\lcm(3,4)=3\cdot12=36.
\]
Hence a prime power $q$ is $h$-admissible if and only if $q\equiv1\Mod{36}$. Moreover, by Lemma \ref{psiFunLem}(2), if $q$ is $h$-admissible, then $\IF_q$ has an index $3$ first-order cyclotomic permutation of cycle type
\[
\gamma_h(q)=x_1x_{\ell((0,1))\frac{q-1}{3h((0,1))}}^{h((0,1))}x_{\ell((2))\frac{q-1}{3h((2))}}^{h((2))}=x_1x_{\frac{2}{9}(q-1)}^3x_{\frac{1}{12}(q-1)}^4.
\]
\end{example}

For the proof of Theorem \ref{mainTheo1}, we are interested in index $d$ first-order cyclotomic permutations of $\IF_q$ which have only long cycles on $\IF_q^{\ast}$ (of length at least $\epsilon q$ for some constant $\epsilon>0$) on $\IF_q^{\ast}$. With regard to those, we have the following result:

\begin{proposition}\label{psiFunProp}
Let $d$ be a positive integer, and let $\epsilon>0$. For each prime power $q$ with $q\equiv1\Mod{d}$, we have that every cycle type of an index $d$ first-order cyclotomic permutation of $\IF_q$ all of whose cycles on $\IF_q^{\ast}$ are of length at least $\epsilon q$ can be written in the form $\gamma_h(q)$ where $h$ is a suitable $q$-admissible $\Sym(d)$-function such that $\max(\ran(h))\leq\epsilon^{-1}$.
\end{proposition}

\begin{proof}
Fix a primitive root $\omega$ of $\IF_q$, let $C$ be the index $d$ subgroup of $\IF_q^{\ast}$, and let $f$ be an index $d$ first-order cyclotomic permutation of $\IF_q$ such that all cycles of $f$ on $\IF_q^{\ast}$ are of length at least $\epsilon q$. Let $\psi$ be the unique permutation in $\Sym(d)$ such that $f(\omega^iC)=\omega^{\psi(i)}C$ for $i=0,1,\ldots,d-1$. By Lemma \ref{psiFunLem}(1), we know that $h_{f,\omega}$ is a $q$-admissible $\psi$-function with
\[
\CT(f)=\gamma_{h_{f,\omega}}(q)=x_1\prod_{\text{cycles }\zeta\text{ of }\psi}{x_{\ell(\zeta)(q-1)/(d\cdot h_{f,\omega}(\zeta))}^{h_{f,\omega}(\zeta)}}.
\]
It follows that the cycle lengths of $f$ on $\IF_q^{\ast}$ are of the form
\[
\ell(\zeta)\frac{q-1}{dh_{f,\omega}(\zeta)}
\]
for the various cycles $\zeta$ of $\psi$. In particular, for each cycle $\zeta$ of $\psi$, we have
\[
\ell(\zeta)\frac{q-1}{dh_{f,\omega}(\zeta)}\geq\epsilon q
\]
and thus
\[
h_{f,\omega}(\zeta)\leq\epsilon^{-1}\cdot\frac{q-1}{q}\cdot\frac{\ell(\zeta)}{d}\leq\epsilon^{-1}\cdot1\cdot 1=\epsilon^{-1},
\]
as required.
\end{proof}

\section{A generalization of a theorem of Carlitz}\label{sec3}

The following theorem, which is due to Carlitz (\cite[Theorem 1]{Car56a}) provides a good amount of control over the multiplicative orders of polynomial evaluations at a fixed element of $\IF_q$:

\begin{theorem}\label{carlitzTheo}(Carlitz, \cite[Theorem 1]{Car56a})
Let $q$ be a prime power, and assume that $Q_1,Q_2,\ldots,Q_r\in\IF_q[T]$ are non-constant, squarefree and pairwise coprime. Fix divisors $e_1,e_2,\ldots,e_r$ of $q-1$, and denote by $N=N(q,Q_1,\ldots,Q_r,e_1,\ldots,e_r)$ the number of $\xi\in\IF_q$ such that for $i=1,2,\ldots,r$, the polynomial evaluation $Q_i(\xi)$ is nonzero and of multiplicative order $e_i$. Then for each $\epsilon>0$, as $q\to\infty$, the number $N$ satisfies the asymptotic formula
\[
N=\frac{\phi(e_1)\phi(e_2)\cdots\phi(e_r)}{(q-1)^r}q+O((\deg{Q_1}+\deg{Q_2}+\cdots+\deg{Q_r})q^{\frac{1}{2}+\epsilon}),
\]
where $\phi$ denotes Euler's totient function.
\end{theorem}

The following is an interesting consequence of Carlitz's asymptotic formula:

\begin{corollary}\label{carlitzCor}
Let $r$ and $d_1,d_2,\ldots,d_r$ as well as $k$ be positive integers. If $q$ is a large enough (relative to $r$, the $d_i$ and $k$) prime power with $q\equiv1\Mod{\lcm(d_1,\ldots,d_r)}$, and if $Q_1,Q_2,\ldots,Q_r\in\IF_q[T]$ are square-free and pairwise coprime polynomials with $\sum_{i=1}^r{\deg{Q_i}}\leq k$, then there is an element $\xi\in\IF_q$ such that for $i=1,2,\ldots,r$, the polynomial evaluation $Q_i(\xi)$ is nonzero and has multiplicative order $\frac{q-1}{d_i}$.
\end{corollary}

\begin{proof}
By Theorem \ref{carlitzTheo}, for every $\epsilon>0$, the number $N$ of these $\xi$ satisfies the asymptotic formula
\[
N=\frac{\phi(\frac{q-1}{d_1})\phi(\frac{q-1}{d_2})\cdots\phi(\frac{q-1}{d_r})}{(q-1)^r}q+O(kq^{\frac{1}{2}+\epsilon})
\]
as $q\to\infty$. The statement of the corollary is clear if we can show that the main term
\[
\frac{\phi(\frac{q-1}{d_1})\phi(\frac{q-1}{d_2})\cdots\phi(\frac{q-1}{d_r})}{(q-1)^r}q
\]
asymptotically dominates over the error term $O(kq^{\frac{1}{2}+\epsilon})$ (for some suitable $\epsilon>0$). Now, using that as $n\to\infty$, the Euler totient function $\phi(n)$ asymptotically dominates over $n^{1-\delta}$ for each fixed $\delta>0$, we see that if $q$ is large enough, the main term satisfies the bounds
\begin{align*}
\frac{\phi(\frac{q-1}{d_1})\phi(\frac{q-1}{d_2})\cdots\phi(\frac{q-1}{d_r})}{(q-1)^r}q &\geq \frac{(\frac{q-1}{d_1})^{1-\delta}(\frac{q-1}{d_2})^{1-\delta}\cdots(\frac{q-1}{d_r})^{1-\delta}}{(q-1)^r}q \\
&=\frac{(q-1)^{-r\delta}}{(d_1\cdots d_r)^{1-\delta}}q\geq\frac{q^{-r\delta}}{d_1\cdots d_r}q=\frac{1}{d_1\cdots d_r}q^{1-r\delta}.
\end{align*}
In particular, the main term will dominate over the error term as long as
\[
1-r\delta>\frac{1}{2}+\epsilon,
\]
which can be achieved e.g.~with $\epsilon:=\frac{1}{6}$ and $\delta:=\frac{1}{4r}$.
\end{proof}

The main goal of this section is to prove the following variant of Theorem \ref{carlitzTheo}, which gives even more control:

\begin{theorem}\label{carlitzGenTheo}
Let $d_1,\ldots,d_r$ and $d$ be positive integers such that $\lcm(d_1,\ldots,d_r)$ divides $d$, and let $j_1,\ldots,j_r$ be integers such that $\gcd(j_i,d)=d_i$ for $i=1,\ldots,r$. Moreover, let $q$ be a prime power with $q\equiv1\Mod{d}$, let $\omega$ be a primitive root of $\IF_q$, and let $Q_1,\ldots,Q_r\in\IF_q[T]$ be non-constant, squarefree and pairwise coprime polynomials. For every $\epsilon>0$, the number of $\xi\in\IF_q$ such that for each $i=1,\ldots,r$, one has
\begin{itemize}
\item $Q_i(\xi)\not=0$,
\item $\ord(Q_i(\xi))=\frac{q-1}{d_i}$, and
\item $\log_{\omega}{Q_i(\xi)}\equiv j_i\Mod{d}$
\end{itemize}
is
\[
c\frac{\phi(\frac{q-1}{d_1})\cdots\phi(\frac{q-1}{d_r})}{(q-1)^r}q+O((\deg{Q_1}+\cdots+\deg{Q_r})q^{\frac{1}{2}+\epsilon})
\]
where
\[
c:=c(d_1,\ldots,d_r,d):=\frac{1}{d^r}\prod_{i=1}^r{\sum_{z_i\mid d}{\frac{\mu(\frac{z_i}{\gcd(z_i,d_i)})^2\phi(z_i)}{\phi(\frac{z_i}{\gcd(z_i,d_i)})^2}}}\in\left[\frac{1}{d^r},1\right],
\]
with $\mu$ denoting the M{\"o}bius function.
\end{theorem}

The proof of Theorem \ref{carlitzGenTheo} is an adaptation of Carlitz's original proof of Theorem \ref{carlitzTheo}. We will need a few lemmas. Let $q$ be a prime power. Recall that a \emph{(multiplicative) character of $\IF_q$} is a function $\chi:\IF_q\rightarrow\IC$ with $\chi(0)=0$ such that the restriction $\chi_{\mid\IF_q^{\ast}}$ is a group homomorphism $\IF_q^{\ast}\rightarrow\IC^{\ast}$. The characters of $\IF_q$ form a cyclic group of order $q-1$ under the character multiplication $\cdot$ defined via $(\chi_1\cdot\chi_2)(\xi):=\chi_1(\xi)\cdot\chi_2(\xi)$ for $\xi\in\IF_q$. The neutral element of this group is the so-called \emph{principal character of $\IF_q$}, denoted by $\chi_0$. It can also be defined by the formula
\[
\chi_0(\xi)=\begin{cases}0, & \text{if }\xi=0, \\ 1, & \text{otherwise}.\end{cases}
\]
For the proof of his theorem, Carlitz introduced the following character sum:

\begin{lemma}\label{carlitzLem}(Carlitz, see \cite[Lemmas 1 and 2]{Car56a})
Let $q$ be a prime power. For a divisor $e$ of $q-1$, consider the character sum
\[
f_e:\IF_q\rightarrow\IC,\xi\mapsto\frac{e}{q-1}\sum_{d\mid e}{\frac{\mu(d)}{d}\sum_{\ord(\chi)\mid d(q-1)/e}{\chi(\xi)}}.
\]
Then the following hold:
\begin{enumerate}
\item For all $\xi\in\IF_q$, we have
\[
f_e(\xi)=\begin{cases}1, & \text{if }\xi\not=0\text{ and the multiplicative order of }\xi\text{ is }e, \\ 0, & \text{otherwise}.\end{cases}
\]
\item The function $f_e$ can also be expressed by the formula
\[
f_e(\xi)=\frac{\phi(e)}{q-1}\sum_{z\mid q-1}{\frac{\mu(\frac{z}{\gcd(z,(q-1)/e)})}{\phi(\frac{z}{\gcd(z,(q-1)/e)})}}\sum_{\ord(\chi)=z}{\chi(\xi)}.
\]
for $\xi\in\IF_q$.
\end{enumerate}
\end{lemma}

Another crucial ingredient of Carlitz's proof is the following lemma:

\begin{lemma}\label{davenportLem}(Davenport and Weil, see \cite[Formula (6) on p.~100]{Dav39a} and \cite{Wei41a}, and also \cite[Lemma 3 and the remarks thereafter]{Car56a})
Let $q$ be a prime power. Let $\chi_1,\ldots,\chi_r$ be non-principal characters of $\IF_q$, and let $Q_1,\ldots,Q_r\in\IF_q[T]$ be non-constant, pairwise coprime squarefree polynomials. Set $\vec{Q}:=(Q_1,\ldots,Q_r)$, $\vec{\chi}:=(\chi_1,\ldots,\chi_r)$, and
\[
S(\vec{Q},\vec{\chi}):=\sum_{\xi\in\IF_q}{\chi_1(Q_1(\xi))\cdots\chi_r(Q_r(\xi))}.
\]
Then
\[
|S(\vec{Q},\vec{\chi})|\leq(\deg{Q_1}+\cdots+\deg{Q_r}-1)q^{1/2}.
\]
\end{lemma}

Actually, Carlitz used a slightly stronger version of Lemma \ref{davenportLem}, which he did not formulate explicitly. We state this result below as Lemma \ref{davenportPlusLem}. Before doing so, we introduce the following notation:

\begin{notation}\label{nullNot}
Let $K$ be a field, and let $P\in K[T]$ be a univariate polynomial over $K$. We denote by $\Null(P)=\Null_K(P)$ the set of all roots of $P$ in $K$. That is,
\[
\Null(P)=\{\xi\in K: P(\xi)=0\}.
\]
\end{notation}

We note that whenever we will use the notation $\Null(P)$ in this paper, we will have $K=\IF_q$.

\begin{lemma}\label{davenportPlusLem}
Let $q$ be a prime power, let $\chi_1,\ldots,\chi_r$ be characters of $\IF_q$, and let $Q_1,\ldots,Q_r\in\IF_q[T]$ be non-constant, pairwise coprime squarefree polynomials. Let $\vec{Q}$, $\vec{\chi}$ and $S(\vec{Q},\vec{\chi})$ be as in Lemma \ref{davenportLem}. The following hold:
\begin{enumerate}
\item If at least one of the $\chi_i$ is non-principal, then
\[
|S(\vec{Q},\vec{\chi})|\leq(\deg{Q_1}+\cdots+\deg{Q_r}-1)q^{\frac{1}{2}}.
\]
\item If all the $\chi_i$ are principal, then
\[
S(\vec{Q},\vec{\chi})=q-|\Null(Q_1Q_2\cdots Q_r)|\geq q-\sum_{i=1}^r{\deg{Q_i}}.
\]
\end{enumerate}
\end{lemma}

\begin{proof}
For statement (1): Assume without loss of generality that $\chi_1,\ldots,\chi_s$, with $s\in\{1,2,\ldots,r\}$, are non-principal and $\chi_{s+1},\ldots,\chi_r$ are principal. Since the principal character $\chi_0$ of $\IF_q$ is constantly $1$ on $\IF_q^{\ast}$ but has the value $0$ in $0_{\IF_q}$, we have that
\begin{align*}
&\chi_1(Q_1(\xi))\cdots\chi_r(Q_r(\xi)) \\
&=\begin{cases}\chi_1(Q_1(\xi))\cdots\chi_s(Q_s(\xi)), & \text{if }\xi\notin\bigcup_{i=s+1}^r{\Null(Q_i)}=\Null(Q_{s+1}\cdots Q_r), \\ 0, & \text{otherwise}.\end{cases}
\end{align*}
Hence
\begin{align*}
|S(\vec{Q},\vec{\chi})| &=\left|\sum_{\xi\in\IF_q}{\chi_1(Q_1(\xi))\cdots\chi_s(Q_s(\xi))} - \sum_{\xi\in\Null(Q_{s+1}\cdots Q_r)}{\chi_1(Q_1(\xi))\cdots\chi_s(Q_s(\xi))}\right| \\
&\leq\left|\sum_{\xi\in\IF_q}{\chi_1(Q_1(\xi))\cdots\chi_s(Q_s(\xi))}\right|+\left|\sum_{\xi\in\Null(Q_{s+1}\cdots Q_r)}{\chi_1(Q_1(\xi))\cdots\chi_s(Q_s(\xi))}\right| \\
&\leq(\deg{Q_1}+\cdots+\deg{Q_s}-1)q^{\frac{1}{2}}+|\Null(Q_{s+1}\cdots Q_r)| \\
&\leq(\deg{Q_1}+\cdots+\deg{Q_s}-1)q^{\frac{1}{2}}+\deg(Q_{s+1}\cdots Q_r) \\
&=(\deg{Q_1}+\cdots+\deg{Q_s}-1)q^{\frac{1}{2}}+\deg{Q_{s+1}}+\cdots+\deg{Q_r} \\
&\leq(\deg{Q_1}+\cdots+\deg{Q_s}-1)q^{\frac{1}{2}}+q^{\frac{1}{2}}\cdot(\deg{Q_{s+1}}+\cdots+\deg{Q_r}) \\
&=(\deg{Q_1}+\cdots+\deg{Q_r}-1)q^{\frac{1}{2}}
\end{align*}
For statement (2): This follows from observing that
\[
\chi_1(Q_1(\xi))\cdots\chi_r(Q_r(\xi))=\begin{cases}1, & \text{if }\xi\notin\bigcup_{i=1}^r{\Null(Q_i)}=\Null(Q_1\cdots Q_r), \\ 0, & \text{otherwise}.\end{cases}
\]
\end{proof}

Since Theorem \ref{carlitzGenTheo} not only deals with multiplicative orders, but also with residue classes of discrete logarithms, we will also need the following character sums:

\begin{lemma}\label{discLogLem}
Let $q$ be a prime power, let $\omega$ be a primitive root of $\IF_q$, let $d$ be a divisor of $q-1$, and let $i\in\IZ$. For $\xi\in\IF_q$, consider the character sum
\[
g_{d,i}(\xi):=\frac{1}{d}\sum_{\ord(\chi)\mid d}{\chi(\omega)^{-i}\chi(\xi)},
\]
where the sum ranges over the multiplicative characters $\chi$ of $\IF_q$ whose order divides $d$. We have
\[
g_{d,i}(\xi)=\begin{cases}1, & \text{if }\xi\not=0\text{ and }\log_{\omega}{\xi}\equiv i\Mod{d}, \\ 0, & \text{otherwise}.\end{cases}
\]
\end{lemma}

\begin{proof}
First of all, since the function values of multiplicative characters of $\IF_q$ at $0_{\IF_q}$ are always $0_{\IC}$, it is clear that $g_{d,i}(0)=0$, whence we may assume that $\xi\not=0$. We denote by $\psi$ some fixed generator of the multiplicative character group of $\IF_q$. Note that for each $\alpha\in\IF_q^{\ast}$, the character value $\psi(\alpha)$ is some complex root of unity of the same order as $\alpha$.

If $\log_{\omega}{\xi}\equiv i\Mod{d}$, then $\omega^{-i}\xi$ lies in the index $d$ subgroup of $\IF_q^{\ast}$, i.e., it is some nonzero $d$-th power $\nu^d$ in $\IF_q$. If $\chi$ is a multiplicative character of $\IF_q$ whose order divides $d$, then $\chi=\psi^e$ for some integer $e$ that is divisible by $\frac{q-1}{d}$. Consequently,
\[
\chi(\omega^{-1}\xi)=\psi^e(\omega^{-1}\xi)=\psi^e(\nu^d)=\psi^{de}(\nu)=\xi_0(\nu)=1.
\]
It follows that
\[
g_{d,i}(\xi)=\frac{1}{d}\sum_{\ord(\chi)\mid d}{\chi(\omega^{-i}\xi)}=\frac{1}{d}\sum_{\ord(\chi)\mid d}{1}=\frac{1}{d}d=1.
\]

Now assume that $\log_{\omega}{\xi}\not\equiv i\Mod{d}$. The multiplicative characters of $\IF_q$ whose order divides $d$ are of the form $\psi^{j\frac{q-1}{d}}$ with $j=0,1,\ldots,d-1$. Hence we can write
\[
g_{d,i}(\xi)=\frac{1}{d}\sum_{j=0}^{d-1}{\psi^{j\frac{q-1}{d}}(\omega^{-i}\xi)}=\frac{1}{d}\sum_{j=0}^{d-1}{\psi(\omega^{-i}\xi)^{j\frac{q-1}{d}}}.
\]
Now, the multiplicative order of $\omega^{-1}\xi$ does not divide $\frac{q-1}{d}$, whence $\psi(\omega^{-1}\xi)^{\frac{q-1}{d}}\not=1$. We can thus apply the geometric series formula to conclude that
\[
g_{d,i}(\xi)=\frac{1}{d}\cdot\frac{\psi(\omega^{-1}\xi)^{d\frac{q-1}{d}}-1}{\psi(\omega^{-1}\xi)^{\frac{q-1}{d}}-1}=\frac{1}{d}\cdot\frac{1-1}{\psi(\omega^{-1}\xi)^{\frac{q-1}{d}}-1}=0,
\]
as required.
\end{proof}

The last lemma which we will need in order to prove Theorem \ref{carlitzGenTheo} is the following, which is well-known:

\begin{lemma}\label{primRootSumLem}
Let $n$ be a positive integer. The sum over all primitive $n$-th roots of unity in $\IC$ is $\mu(n)$.
\end{lemma}

\begin{proof}[Proof of Theorem \ref{carlitzGenTheo}]
By Lemmas \ref{carlitzLem}(1) and \ref{discLogLem}, the number of $\xi\in\IF_q$ we are interested in is
\[
\sum_{\xi\in\IF_q}{f_{\frac{q-1}{d_1}}(Q_1(\xi))\cdots f_{\frac{q-1}{d_r}}(Q_r(\xi))g_{d,j_1}(Q_1(\xi))\cdots g_{d,j_r}(Q_r(\xi))}.
\]
Using Lemma \ref{carlitzLem}(2) and the definition of $g_{d,i}(\xi)$ in Lemma \ref{discLogLem}, we can rewrite this into
\begin{equation}\label{rewrittenSumEq}
\frac{\prod_{i=1}^r{\phi(\frac{q-1}{d_i})}}{(q-1)^rd^r}\sum_{z_1,\ldots,z_r\mid q-1}{\prod_{i=1}^r{\frac{\mu(\frac{z_i}{\gcd(z_i,d_i)})}{\phi(\frac{z_i}{\gcd(z_i,d_i)})}}\sum_{\psi_1,\ldots,\psi_r: \atop \ord(\psi_i)\mid d}{\prod_{i=1}^r{\psi_i(\omega)^{-j_i}}\sum_{\chi_1,\ldots,\chi_r: \atop \ord(\chi_i)=z_i}{\overline{S}(\vec{Q},\vec{\chi},\vec{\psi})}}}
\end{equation}
where
\[
\overline{S}(\vec{Q},\vec{\chi},\vec{\psi}):=S((Q_1,\ldots,Q_r),(\chi_1\psi_1,\ldots,\chi_r\psi_r))=\sum_{\xi\in\IF_q}{\prod_{i=1}^r{(\chi_i\psi_i)(Q_i(\xi))}}.
\]
Observe that by Lemma \ref{davenportPlusLem}, we have an interesting dichotomy here: The complex number $\overline{S}(\vec{Q},\vec{\chi},\vec{\psi})$ is \enquote{large} (of the form $q-|\Null(\prod_{i=1}^r{Q_i})|$) if all the characters $\chi_i\psi_i$ for $i=1,2,\ldots,r$ are principal, and otherwise, it is \enquote{small} (of absolute value less than $\deg(Q_1\cdots Q_r)q^{1/2}$). Our strategy is to show that the former case accounts for the main term
\[
c\frac{\phi(\frac{q-1}{d_1})\cdots\phi(\frac{q-1}{d_r})}{(q-1)^r}q
\]
in our counting formula, whereas the latter case yields the error term
\[
O((\deg{Q_1}+\cdots+\deg{Q_r})q^{\frac{1}{2}+\epsilon}).
\]
We start by analyzing the case where $\overline{S}(\vec{Q},\vec{\chi},\vec{\psi})$ is large. Assume that the indices $z_1,\ldots,z_r\mid q-1$ of the outermost sum in formula (\ref{rewrittenSumEq}) are fixed. Then in the second-innermost sum, the orders of the $\chi_i$ are fixed (as the $z_i$). Now, $\chi_i\psi_i=\chi_0$ is equivalent to $\psi_i=\chi_i^{-1}$, a necessary condition for which is that $z_i=\ord(\chi_i)=\ord(\psi_i)\mid d$. Conversely, if the $\psi_i$ are chosen such that $\ord(\psi_i)=z_i\mid d$ for $i=1,2,\ldots,r$, then there is precisely one choice for the $\chi_i$ (namely $\chi_i:=\psi_i^{-1}$) such that the $\chi_i\psi_i$ are all principal. Applying Lemma \ref{davenportPlusLem}(2), we see that this corresponds to the subsum
\begin{equation}\label{mainTermEq}
\frac{\prod_{i=1}^r{\phi(\frac{q-1}{d_i})}}{(q-1)^rd^r}\sum_{z_1,\ldots,z_r\mid d}{\prod_{i=1}^r{\frac{\mu(\frac{z_i}{\gcd(z_i,d_i)})}{\phi(\frac{z_i}{\gcd(z_i,d_i)})}}}\sum_{\psi_1,\ldots,\psi_r:\ord(\psi_i)=z_i}{\prod_{i=1}^r{\psi_i(\omega)^{-j_i}}\left(q-|\Null(\prod_{i=1}^r{Q_i})|\right)}.
\end{equation}
In order to transform this further, we pull the constant factor $\left(q-|\Null(\prod_{i=1}^r{Q_i})|\right)$ out in front and observe that
\[
\sum_{\psi_1,\ldots,\psi_r:\ord(\psi_i)=z_i}{\prod_{i=1}^r{\psi_i(\omega)^{-j_i}}}=\prod_{i=1}^r{\sum_{\ord(\psi_i)=z_i}{\psi_i(\omega)^{-j_i}}}.
\]
Fix $i\in\{1,\ldots,r\}$. Note that as $\psi_i$ ranges over the multiplicative characters of $\IF_q$ of order $z_i$, the character value $\psi_i(\omega)$ ranges over the primitive $z_i$-th roots of unity in $\IC$. By our assumption that $\gcd(j_i,d)=d_i$ and the fact that $z_i\mid d$, we conclude that $\gcd(j_i,z_i)=\gcd(d_i,z_i)$. Hence, as $\psi_i$ ranges over the multiplicative characters of $\IF_q$ of order $z_i$, the power $\psi_i(\omega)^{-j_i}$ runs $\frac{\phi(z_i)}{\phi(\frac{z_i}{\gcd(z_i,d_i)})}$ times through the primitive $\frac{z_i}{\gcd(z_i,d_i)}$-th roots of unity in $\IC$. It follows by Lemma \ref{primRootSumLem} that
\[
\sum_{\ord(\psi_i)=z_i}{\psi_i(\omega)^{-j_i}}=\frac{\phi(z_i)}{\phi(\frac{z_i}{\gcd(z_i,d_i)})}\mu(\frac{z_i}{\gcd(z_i,d_i)}).
\]
Substituting this into formula (\ref{mainTermEq}), we get that the expression in that formula is equal to
\begin{align*}
&\frac{\prod_{i=1}^r{\phi(\frac{q-1}{d_i})}}{(q-1)^rd^r}\left(q-|\Null(\prod_{i=1}^r{Q_i})|\right)\sum_{z_1,\ldots,z_r\mid d}{\prod_{i=1}^r{\frac{\mu(\frac{z_i}{\gcd(z_i,d_i)})^2\phi(z_i)}{\phi(\frac{z_i}{\gcd(z_i,d_i)})^2}}} \\
&=\frac{\prod_{i=1}^r{\phi(\frac{q-1}{d_i})}}{(q-1)^rd^r}\left(q-|\Null(\prod_{i=1}^r{Q_i})|\right)\cdot d^rc = c\frac{\prod_{i=1}^r{\phi(\frac{q-1}{d_i})}}{(q-1)^r}q+O(\sum_{i=1}^r{\deg{Q_i}}).
\end{align*}

Now we consider those summands in formula (\ref{rewrittenSumEq}) where $\overline{S}(\vec{Q},\vec{\chi},\vec{\psi})$ is small. Using Lemma \ref{davenportPlusLem}(1), we see that the absolute value of the sum of all those summands is at most
\begin{align*}
&\frac{\prod_{i=1}^r{\phi(\frac{q-1}{d_i})}}{(q-1)^rd^r}\sum_{z_1,\ldots,z_r\mid q-1}{\prod_{i=1}^r{\frac{1}{\phi(\frac{z_i}{\gcd(z_i,d_i)})}}\sum_{\psi_1,\ldots,\psi_r: \atop \ord(\psi_i)\mid d}{\sum_{\chi_1,\ldots,\chi_r: \atop \ord(\chi_i)=z_i}{\left(\sum_{i=1}^r{\deg{Q_i}}\right)q^{1/2}}}} \\
&=\frac{\prod_{i=1}^r{\phi(\frac{q-1}{d_i})}}{(q-1)^rd^r}\left(\sum_{i=1}^r{\deg{Q_i}}\right)q^{1/2}\cdot\sum_{z_1,\ldots,z_r\mid q-1}{\frac{d^r\prod_{i=1}^r{\phi(z_i)}}{\prod_{i=1}^r{\phi(\frac{z_i}{\gcd(z_i,d_i)})}}} \\
&\leq\frac{\prod_{i=1}^r{\phi(\frac{q-1}{d_i})}}{(q-1)^r}\left(\sum_{i=1}^r{\deg{Q_i}}\right)q^{1/2}\sum_{z_1,\ldots,z_r\mid q-1}{\prod_{i=1}^r{\gcd(z_i,d_i)}} \\
&\leq\frac{\prod_{i=1}^r{\frac{q-1}{d_i}d_i}}{(q-1)^r}\left(\sum_{i=1}^r{\deg{Q_i}}\right)q^{1/2}\cdot\tau(q-1)^r=\left(\sum_{i=1}^r{\deg{Q_i}}\right)q^{1/2}\cdot\tau(q-1)^r \\
&\leq\left(\sum_{i=1}^r{\deg{Q_i}}\right)q^{\frac{1}{2}+\epsilon}.
\end{align*}
where the first inequality uses the bound $\phi(n\cdot c)\leq\phi(n)\cdot c$, valid for all positive integers $n$ and $c$, and the last inequality uses that the number of divisors function $\tau$ grows more slowly than any power function $n\mapsto n^{\epsilon}$.

It remains to prove the bounds $\frac{1}{d^r}\leq c\leq 1$, which are equivalent to
\begin{equation}\label{boundsEq}
1\leq\prod_{i=1}^r{\sum_{z_i\mid d}{\frac{\mu(\frac{z_i}{\gcd(z_i,d_i)})^2\phi(z_i)}{\phi(\frac{z_i}{\gcd(z_i,d_i)})^2}}}\leq d^r.
\end{equation}
For $i=1,2,\ldots,r$, define
\[
h_i:\IN^+\rightarrow\IQ, n\mapsto\sum_{z\mid n}{\frac{\mu(\frac{z}{\gcd(z,d_i)})^2\phi(z)}{\phi(\frac{z}{\gcd(z,d_i)})^2}}.
\]
Then formula (\ref{boundsEq}) simplifies to
\[
1\leq\prod_{i=1}^r{h_i(d)}\leq d^r,
\]
from which we see that it suffices to show that for each $i=1,2,\ldots,r$, we have $1\leq h_i(d)\leq d$. For the lower bound \enquote{$h_i(d)\geq 1$} observe that $h_i(d)$ is a sum of non-negative real numbers, whence it is bounded from below by its summand for $z=1$, which is
\[
\frac{\mu(\frac{1}{1})^2\phi(1)}{\phi(\frac{1}{1})^2}=\frac{1}{1}=1,
\]
as required. For the upper bound \enquote{$h_i(d)\leq d$}, note that $h_i$ is a multiplicative function, whence it suffices to show that $h_i(\ell^k)\leq\ell^k$ for each prime $\ell$ and each positive integer $k$. We do so in a case distinction:
\begin{enumerate}
\item Case: $\ell$ does not divide $d_i$. Then
\[
h_i(\ell^k)=\sum_{j=0}^k{\frac{\mu(\ell^j)^2\phi(\ell^j)}{\phi(\ell^j)^2}}=1+\frac{1}{\ell-1}\leq 2\leq\ell^k,
\]
where the last equality uses that all summands for $j>1$ are zero due to the M{\"o}bius function factor.
\item Case: $\ell$ divides $d_i$. Then
\begin{align*}
h_i(\ell^k) &=\sum_{j=0}^k{\frac{\mu(\ell^{\max(0,j-\nu_{\ell}(d_i))})^2\phi(\ell^j)}{\phi(\ell^{\max(0,j-\nu_{\ell}(d_i))})^2}} \\
&=\begin{cases}1+\phi(\ell)+\phi(\ell^2)+\cdots+\phi(\ell^{\nu_{\ell}(d_i)})+\frac{\phi(\ell^{\nu_{\ell}(d_i)+1})}{(\ell-1)^2}, & \text{if }k>\nu_{\ell}(d_i), \\ 1+\phi(\ell)+\phi(\ell^2)+\cdots+\phi(\ell^k), & \text{if }k\leq\nu_{\ell}(d_i).\end{cases}
\end{align*}
Either way, we have that $h_i(\ell^k)\leq\ell^k$: If $k>\nu_{\ell}(d_i)$, then
\begin{align*}
&1+\phi(\ell)+\phi(\ell^2)+\cdots+\phi(\ell^{\nu_{\ell}(d_i)})+\frac{\phi(\ell^{\nu_{\ell}(d_i)+1})}{(\ell-1)^2}=1+(\ell-1)\sum_{i=0}^{\nu_{\ell}(d_i)-1}{\ell^i}+\frac{\ell^{\nu_{\ell}(d_i)}}{\ell-1} \\
&=1+(\ell-1)\frac{\ell^{\nu_{\ell}(d_i)}-1}{\ell-1}+\frac{\ell^{\nu_{\ell}(d_i)}}{\ell-1}\leq\ell^{\nu_{\ell}(d_i)}+\ell^{\nu_{\ell}(d_i)}=2\ell^{\nu_{\ell}(d_i)}\leq\ell^{\nu_{\ell}(d_i)+1}\leq\ell^k,
\end{align*}
and if $k\leq\nu_{\ell}(d_i)$, then
\[
1+\phi(\ell)+\phi(\ell^2)+\cdots+\phi(\ell^k)=1+(\ell-1)\sum_{i=0}^{k-1}{\ell^i}=1+(\ell-1)\frac{\ell^k-1}{\ell-1}=\ell^k.
\]
\end{enumerate}
\end{proof}

We conclude this section by noting that Theorem \ref{carlitzGenTheo} implies the following, the proof of which is analogous to the one of Corollary \ref{carlitzCor}:

\begin{corollary}\label{carlitzGenCor}
Let $d_1,\ldots,d_r$ and $d$ be positive integers such that $\lcm(d_1,\ldots,d_r)$ divides $d$. Moreover, let $0<\delta<\frac{1}{2}$. There is a constant $C_2=C_2(d_1,\ldots,d_r,d,\delta)$ such that if
\begin{itemize}
\item $q$ is a prime power with $q\equiv1\Mod{d}$ and $q\geq C_2$,
\item $\omega$ is a primitive root of $\IF_q$,
\item $Q_1,\ldots,Q_r\in\IF_q[T]$ are non-constant, squarefree and pairwise coprime polynomials that satisfy $\sum_{i=1}^r{\deg{Q_i}}\leq q^{\delta}$, and
\item $j_1,\ldots,j_r$ are integers with $\gcd(j_i,d)=d_i$ for $i=1,2,\ldots,r$,
\end{itemize}
then there is an element $\xi\in\IF_q$ such that for $i=1,2,\ldots,r$, the following hold:
\begin{enumerate}
\item $Q_i(\xi)\not=0$;
\item the multiplicative order of $Q_i(\xi)$ is $\frac{q-1}{d_i}$;
\item $\log_{\omega}(Q_i(\xi))\equiv j_i\Mod{d}$.
\end{enumerate}
\end{corollary}

\section{Proof of Theorem \ref{mainTheo1}}\label{sec4}

In Section \ref{sec2}, we introduced the concept of a $\Sym(d)$-function as a means to enumerate cycle types of index $d$ first-order cyclotomic permutations of $\IF_q$. We will now see that using Corollary \ref{carlitzCor}, one can prove a partial strengthening of Lemma \ref{psiFunLem}(1): If $\psi\in\Sym(d)$ is of a particular form (see Definition \ref{specialPermDef}, of a \enquote{special permutation}), then for each $\psi$-function $h$, if $q$ is a large enough $h$-admissible prime power and $\omega$ is a primitive root of $\IF_q$, not only does there exist an index $d$ first-order cyclotomic permutation $f$ of $\IF_q$ with $h_{f,\omega}=h$, but $f$ can also be chosen such that a fixed number of additive translates $f+c_1\id,f+c_2\id,\ldots,f+c_n\id$ (where $c_1,c_2,\ldots,c_n\in\IF_q$) are also permutations of $\IF_q$.

\begin{definition}\label{specialPermDef}
Let $d$ be a positive integer, and let $\psi\in\Sym(d)$. Assume that the cycle lengths of $\psi$ are $\ell_1\geq\ell_2\geq\cdots\geq\ell_r$, listed with multiplicities (so that the sum of the $\ell_j$ is $d$). Then $\psi$ is called \emph{special} if it is of the form
\[
\psi=\prod_{i=1}^r{\left(\sum_{j=1}^{i-1}{\ell_j},\sum_{j=1}^{i-1}{\ell_j}+1,\ldots,\sum_{j=1}^{i-1}{\ell_j}+(\ell_i-1)\right)}.
\]
\end{definition}

We note that the special permutations in $\Sym(d)$ form a set of conjugacy class representatives.

\begin{lemma}\label{psiFunLem2}
Let $d$ and $n$ be positive integers, let $\psi\in\Sym(d)$ be special, and let $h$ be a $\psi$-function. There is a $q_2=q_2(h,n)>0$ such that for all $h$-admissible prime powers $q$ with $q\geq q_2$ and all $c_1,c_2,\ldots,c_n\in\IF_q^{\ast}$, there is a first-order cyclotomic permutation $f$ of $\IF_q$ of index $d$ such that the following hold:
\begin{enumerate}
\item For $j=1,2,\ldots,n$, the function $f+c_j\id$ is a (first-order and index $d$ cyclotomic) permutation of $\IF_q$ that stabilizes each coset of the index $d$ subgroup $C$ of $\IF_q^{\ast}$.
\item The cycle type of $f$ is $\gamma_h(q)$.
\end{enumerate}
\end{lemma}

\begin{proof}
If $\zeta$ is a cycle of $\sigma$ of length $\ell$, we let $i_{\init}(\zeta)$ denote the smallest element of $\{0,1,\ldots,d-1\}$ that lies on $\zeta$, so that $\zeta=(i_{\init}(\zeta),i_{\init}(\zeta)+1,\ldots,i_{\init}(\zeta)+\ell-1)$, and we set $i_{\term}(\zeta):=i_{\init}(\zeta)+\ell-1$. We call $i_{\init}(\zeta)$ and $i_{\term}(\zeta)$ the \emph{initial} and \emph{terminal index of $\zeta$}, respectively. Moreover, for $i=0,1,\ldots,d-1$, let $\zeta(i)$ be the cycle of $i$ under $\sigma$. Let $q$ be an $h$-admissible prime power. Until further notice, let $\omega$ be an arbitrary but fixed primitive root of $\IF_q$.

Denoting by $\ell(\zeta)$ the length of a cycle $\zeta$ of $\psi$, and for $i=0,1,\ldots,d-1$ by $\zeta(i)$ the unique cycle of $\psi$ such that $i\in\supp(\zeta(i))$, we define a function $f_{\omega}:\IF_q\rightarrow\IF_q$ via
\[
f_{\omega}(x):=\begin{cases}0, & \text{if }x=0, \\ \omega x, & \text{if }x\in\omega^iC,i\not=i_{\term}(\zeta(i)), \\ \omega^{-\ell(\zeta(i))+1+dh(\zeta(i))}x, & \text{if }x\in\omega^iC,i=i_{\term}(\zeta(i)).\end{cases}
\]
Recall our definition of the field element $\pi_{f,\omega}(\zeta)$, where $\zeta$ is a cycle of $\psi$, from Notation \ref{piNot}. By definition of $f_{\omega}$, we find that for each cycle $\zeta$ of $\psi$, one has
\[
\pi_{f,\omega}(\zeta)=\omega^{\ell(\zeta)-1}\cdot\omega^{-\ell(\zeta)+1+dh(\zeta)}=\omega^{dh(\zeta)},
\]
whence
\[
\ord(\pi_{f,\omega})(\zeta)=\frac{q-1}{dh(\zeta)},
\]
or equivalently
\[
h(\zeta)=\frac{q-1}{d\ord(\pi_{f,\omega})}=h_{f,\omega}(\zeta).
\]
This proves that $h_{f,\omega}=h$. It remains to show that if $q$ is large enough (relative to $h$ and $n$), then for arbitrarily chosen elements $c_1,c_2,\ldots,c_n\in\IF_q^{\ast}$, the primitive root $\omega$ can be suitably chosen such that for $j=1,2,\ldots,n$, the function $f_{\omega}+c_j\id$ is a permutation of $\IF_q$ that stabilizes each coset of $C$ in $\IF_q^{\ast}$.

Since we certainly have $(f_{\omega}+c_j\id)(0)=0$ for $j=1,2,\ldots,n$, we will fix $i\in\{0,1,\ldots,d-1\}$ and focus on the mapping behavior of the translates $f_{\omega}+c_j\id$ on the coset $C_i:=\omega^iC$ of $C$ in $\IF_q^{\ast}$. By definition of $f_{\omega}$, for all $x\in C_i$ and $j=1,2,\ldots,n$, we have
\[
(f_{\omega}+c_j\id)(x)=f_{\omega}(x)+c_jx=\begin{cases}(\omega+c_j)x, & \text{if }i\not=i_{\term}(\zeta(i)), \\ (\omega^{-\ell(\zeta(i))+1+dh(\zeta(i))}+c_j)x, & \text{if }i=i_{\term}(\zeta(i)).\end{cases}
\]
It follows that as long as the constants
\begin{itemize}
\item $\omega+c_j$ for $i$ not terminal in $\zeta(i)$ and
\item $\omega^{-\ell(\zeta(i))+1+dh(\zeta(i))}+c_j$ for $i$ terminal in $\zeta(i)$
\end{itemize}
are always in $C$, the function $f_{\omega}+c_j\id$ will be a permutation of $\IF_q$ that stabilizes every coset of $C$ (more precisely, its restriction to a coset of $C$ is always a multiplication by a constant element of $C$). Proving that we can indeed choose $\omega$ such that those constants are in $C$ involves an application of Corollary \ref{carlitzCor}, as we will now explain.

For $i=0,1,\ldots,d-1$ and $j=1,2,\ldots,n$, we define the following univariate polynomial $P_{i,j}\in\IF_q[T]$:
\[
P_{i,j}:=\begin{cases}T+c_j, & \text{if }i\text{ is not the terminal index of }\zeta(i), \\ T^{-\ell(\zeta(i))+1+dh(\zeta(i))}+c_j, & \text{if }i\text{ is the terminal index of }\zeta(i).\end{cases}
\]
Observe that each of the constants which we want to force to be in $C$ (as explained in the previous paragraph) is of the form $P_{i,j}(\omega)$ for suitable $i$ and $j$, and so we want to force these polynomial evaluations to lie in $C$, for a suitable primitive root $\omega$ of $\IF_q$.

Now, let $Q_1,\ldots,Q_r\in\IF_q[T]$ be those irreducible polynomials that divide at least one of the $P_{i,j}$, listed without repetitions. Moreover, set $Q_0:=T$, and observe that $Q_0$ is distinct from any $Q_k$ with $k>0$ (this uses that the $c_j$ are all nonzero). Note that for all $i$ and $j$,
\[
\deg{P_{i,j}}\leq d\max{\ran(h)},
\]
and thus
\[
r\leq\sum_{i,j}{\deg{P_{i,j}}}\leq d^2n\max{\ran(h)}.
\]
Hence
\[
\sum_{k=0}^r{\deg{Q_k}}\leq1+r\cdot d\max{\ran(h)}\leq 1+d^3n(\max{\ran(h)})^2.
\]
This shows that both the number $r+1$ of the pairwise coprime irreducible polynomials $Q_k$ and their degree sum are bounded in terms of $d$, $n$ and $h$. But $d$ is bounded in terms of $h$ (it is the sum of the lengths of the cycles from $\codom(h)$), so the polynomial number and degree sum are actually bounded in terms of $n$ and $h$ alone. Therefore, Corollary \ref{carlitzCor} implies that if $q$ is large enough relative to $d$, there is an $\omega\in\IF_q$ such that $Q_k(\omega)\not=0$ for $k=0,1,\ldots,r$ and
\[
\ord(Q_0(\omega))=\ord(\omega)=q-1\text{ and }\ord(Q_k(\omega))=\frac{q-1}{d}\text{ for }k=1,2,\ldots,r.
\]
In particular, $\omega$ is a primitive root of $\IF_q$, and since $Q_k(\omega)\in C$ for $k=1,2,\ldots,r$ and each evaluation $P_{i,j}(\omega)$ is a product of such evaluations $Q_k(\omega)$, it follows that $P_{i,j}(\omega)\in C$ for all $i$ and $j$, as required.
\end{proof}

\begin{example}\label{psiFunEx2}
Set $d:=3$, $n:=1$, $\psi:=(0,1)(2)\in\Sym(3)$, and let $h$ be the $\psi$-function with $h((0,1))=3$ and $h((2))=4$. Recall from Example \ref{psiFunEx} that a prime power $q$ is $h$-admissible if and only if $q\equiv1\Mod{36}$, in which case
\[
\gamma_h(q)=x_1x_{\frac{2}{9}(q-1)}^3x_{\frac{1}{12}(q-1)}^4.
\]
Hence Lemma \ref{psiFunLem2} guarantees that as long as $q$ is a large enough prime power with $q\equiv1\Mod{36}$, there will always be an index $3$ first-order cyclotomic permutation $f$ of $\IF_q$ of cycle type
\[
x_1x_{\frac{2}{9}(q-1)}^3x_{\frac{1}{12}(q-1)}^4
\]
such that $f+1_{\IF_q}\id=f+\id$ is a permutation of $\IF_q$ (i.e., such that $f$ is a complete mapping of $\IF_q$).
\end{example}

Now that we have Proposition \ref{psiFunProp} and Lemma \ref{psiFunLem2}, the actual proof of Theorem \ref{mainTheo1} is relatively easy:

\begin{proof}[Proof of Theorem \ref{mainTheo1}]
Let $\Mcal$ be the set of all $\Sym(d)$-functions $h$ such that
\[
\max(\ran(h))\leq\epsilon^{-1}.
\]
Then $\Mcal$ is a finite set of $\Sym(d)$-functions, depending on $d$ and $\epsilon$. Set
\[
q_0=q_0(d,n,\epsilon):=\max_{h\in\Mcal,m\leq n}{q_2(h,m)}>0
\]
where $q_2$ is as in Lemma \ref{psiFunLem2}. We claim that this definition of $q_0$ does the job. Indeed, let $q\geq q_0$ be a prime power. We may assume that $q\equiv1\Mod{d}$, since otherwise, the statement of the theorem is vacuously true for $q$ (because there are no index $d$ first-order cyclotomic permutations on $\IF_q$). Let $\omega$ be a primitive root of $\IF_q$, let $c_1,c_2,\ldots,c_n\in\IF_q$ be arbitrary, and let $\gamma$ be the cycle type of some index $d$ first-order cyclotomic permutation on $\IF_q$ all of whose cycles on $\IF_q^{\ast}$ are of length at least $\epsilon q$. Let $c'_1,\ldots,c'_m$ be the distinct nonzero field elements among $c_1,\ldots,c_n$. By the definition of $q_0$, there exists an index $d$ first-order cyclotomic permutation $f$ on $\IF_q$ such that
\begin{itemize}
\item $\CT(f)=\gamma_h(q)=\gamma$ and
\item $f+c'_j\id$ is a permutation of $\IF_q$ for $j=1,2,\ldots,m$.
\end{itemize}
This $f$ has the properties we desire, which are
\begin{itemize}
\item $\CT(f)=\gamma_h(q)=\gamma$ and
\item $f+c_j\id$ is a permutation of $\IF_q$ for $j=1,2,\ldots,n$.
\end{itemize}
\end{proof}

\section{Proof of Theorem \ref{mainTheo2}}\label{sec5}

We may assume without loss of generality that $c_1=0_{\IF_q}$ (and $c_2,\ldots,c_n$ are pairwise distinct nonzero elements of $\IF_q$). Fix a primitive root $\omega$ of $\IF_q$, and denote by $C_i:=\omega^iC$ for $i=0,1,\ldots,d-1$ the distinct cosets of $C$ in $\IF_q^{\ast}$. For $j=1,2,\ldots,n$, denote by $\sigma_j$ the unique function $\{0,1,\ldots,d-1\}\rightarrow\{0,1,\ldots,d-1\}$ such that $s_j(C_i)=C_{\sigma_j(i)}$ for all $i=0,1,\ldots,d-1$.

For $i=0,1,\ldots,d-1$, set $\delta_i:=\sigma_1(i)-i$ and $b_i:=\omega^{\delta_i}$. For $\vec{\gamma}=(\gamma_0,\ldots,\gamma_{d-1})\in C^{d-1}$, set
\[
f_{\vec{\gamma}}(x):=\begin{cases}0, & \text{if }x=0, \\ b_i\gamma_i x, & \text{if }x\in C_i\text{ for some }i=0,1,\ldots,d-1.\end{cases}
\]
Then for each $i=0,1,\ldots,d-1$, we have
\[
f_{\vec{\gamma}}(C_i)=b_i\gamma_iC_i=\omega^{\sigma_1(i)-i}\gamma_i\omega^iC=\omega^{\sigma_1(i)}C=C_{\sigma_1(i)}=s_1(C_i).
\]
It remains to show that if $q$ is large enough (relative to $d$ and $n$), then the vector $\vec{\gamma}\in C^d$ can be adjusted such that we also have
\[
(f_{\vec{\gamma}}+c_j\id)(C_i)=s_j(C_i)
\]
for $i=0,1,\ldots,d-1$ and $j=2,3,\ldots,n$. Now, for a fixed $i\in\{0,1,\ldots,d-1\}$, we have that the (linear) polynomials
\[
T,b_iT+c_2,b_iT+c_3,\ldots,b_iT+c_n\in\IF_q[T]
\]
are non-constant, squarefree and pairwise coprime. Applying Corollary \ref{carlitzGenCor} with $r:=n$, with
\[
d_k:=\begin{cases}d, & \text{if }k=1, \\ \gcd(\sigma_k(i)-i,d), & \text{if }k=2,3,\ldots,n,\end{cases}
\]
and with
\[
j_k:=\begin{cases}0, & \text{if }k=1, \\ \sigma_k(i)-i, & \text{if }k=2,3,\ldots,n,\end{cases}
\]
we conclude that if $q$ is large enough (relative to $d$ and $n$), then there is a $\gamma_i\in\IF_q$ such that
\begin{enumerate}
\item $\gamma_i\not=0$,
\item the multiplicative order of $\gamma_i$ is $\frac{q-1}{d}$, and
\item for $j=2,3,\ldots,n$, one has $b_i\gamma_i+c_j\not=0$ and $\log_{\omega}(b_i\gamma_i+c_j)\equiv\sigma_j(i)-i\Mod{d}$.
\end{enumerate}
In particular, $\gamma_i$ is a generator of $C$, whence $\vec{\gamma}:=(\gamma_0,\ldots,\gamma_{d-1})\in C^d$. By construction, for $i=0,1,\ldots,d-1$ and $j=2,3,\ldots,n$, we have
\[
(f_{\vec{\gamma}}+c_j\id)(x)=b_i\gamma_ix+c_jx=(b_i\gamma_i+c_j)x,
\]
whence
\[
(f_{\vec{\gamma}}+c_j\id)(C_i)=(b_i\gamma_i+c_j)C_i=C_{\sigma_j(i)-i+i}=C_{\sigma_j(i)}=s_j(C_i),
\]
as required.

\section{Proof of Theorem \ref{mainTheo3}}\label{sec6}

In addition to the results and ideas discussed in Section \ref{sec3}, we will need the following number-theoretic lemma:

\begin{lemma}\label{uChoiceLem}
Let $d$ be a positive integer, and let $\epsilon>0$. There is a constant $C_1=C_1(d,\epsilon)$ such that if $q=p^f$ is a prime power with $q\equiv1\Mod{d}$ and $q\geq C_1$, then there is an odd unit $u\in(\IZ/\frac{q-1}{d}\IZ)^{\ast}$ such that $1<u<\frac{q^{\epsilon}}{d}$ and $u\not\equiv e(d-1)\Mod{p}$ where $e$ denotes the multiplicative inverse of $d$ modulo $p$.
\end{lemma}

\begin{proof}
We may assume that $\epsilon<1$. If $q$ is large enough (relative to $d$), then there is an odd unit $v\in(\IZ/\frac{q-1}{d}\IZ)^{\ast}$ such that $1<v<\frac{q^{\epsilon/2}}{\sqrt{d}}$ -- just let $v$ be any of the primes less than $\frac{q^{\epsilon/2}}{\sqrt{d}}$ that are odd and do not divide $\frac{q-1}{d}$. Such a $v$ exists because by the Prime Number Theorem, if $q$ is large enough, then the total number of primes less than $\frac{q^{\epsilon/2}}{\sqrt{d}}$ is at least
\[
\frac{q^{\epsilon/2}/\sqrt{d}}{2\log(q^{\epsilon/2}/\sqrt{d})}=\frac{q^{\epsilon/2}}{2\sqrt{d}(\frac{\epsilon}{2}\log{q}-\frac{1}{2}\log{d})},
\]
whereas the number of \enquote{forbidden} primes, those that are equal to $2$ or divide $\frac{q-1}{d}$, is at most
\[
1+\log_2\left(\frac{q-1}{d}\right)=1+\frac{\log(q-1)-\log{d}}{\log{2}}.
\]
Note that each such $v$ is also a non-involution in $(\IZ/\frac{q-1}{d}\IZ)^{\ast}$, because its square satisfies the inequality chain
\[
1<v^2<\frac{q^{\epsilon}}{d}<\frac{q-1}{d}.
\]
Hence, if the statement of the lemma was false, we would have
\[
v\equiv e(d-1)\equiv v^2\equiv e^2(d-1)^2\Mod{p}.
\]
Because $0$ and $1$ are the only solutions modulo $p$ of the congruence $x^2\equiv x\Mod{p}$, this allows us to conclude that $e(d-1)\equiv0,1\Mod{p}$. Using that $e$ is the multiplicative inverse of $d$ modulo $p$, we see that this is equivalent to $e\equiv 0,1\Mod{p}$. But $e$ is a unit modulo $p$, so $e\equiv0\Mod{p}$ is impossible, and we conclude that $d\equiv e\equiv1\Mod{p}$ and $e(d-1)\equiv0\Mod{p}$. Thus, all we need to do in order to obtain a contradiction is to find an odd unit $u$ with $1<u<\frac{q^{\epsilon}}{d}$ such that $u\not\equiv0\Mod{p}$. For this, we can choose $u$ to be any of the roughly
\[
\frac{q^{\epsilon}}{d(\epsilon\log{q}-\log{d})}-(\epsilon\log{q}-\log{d})-2
\]
odd primes less than $\frac{q^{\epsilon}}{d}$ that are distinct from $p$ and do not divide $\frac{q-1}{d}$.
\end{proof}

Let us now turn to the actual proof of Theorem \ref{mainTheo3}. We may assume without loss of generality that $d$ is a prime. Indeed, assume that the statement holds for primes. Then if $d$ is an arbitrary positive integer with $d>1$, and if $\ell$ is the smallest prime divisor of $d$, we have that every prime power $q$ with $q\equiv1\Mod{d}$ is also congruent to $1$ modulo $\ell$. Therefore, if $q\equiv1\Mod{d}$ and $q\geq C(\ell)$, the field $\IF_q$ admits a complete mapping with the desired properties, which shows that the statement of the theorem holds with $C(d):=C(\ell)$.

Assuming henceforth that $d$ is a prime, we choose (by Lemma \ref{uChoiceLem}) an odd unit $u\in(\IZ/\frac{q-1}{d}\IZ)^{\ast}$ with $1<u<\frac{q^{\epsilon}}{d}$ (where $\epsilon$ is number in $\left(0,1\right)$ whose precise value will be specified later) such that $u\not\equiv e(d-1)\Mod{p}$, where $e$ is the multiplicative inverse of $d$ modulo $p$. Set $o:=u\cdot d-(d-1)$. For a fixed primitive root $\omega$ of $\IF_q$ and associated indexing $C_i:=\omega^iC$ (for $i=0,1,\ldots,d-1$) of the cosets in $\IF_q^{\ast}$ of the index $d$ subgroup $C$ of $\IF_q^{\ast}$, we consider the first-order cyclotomic mapping $f:\IF_q\rightarrow\IF_q$ of the following form:
\[
f(x)=\begin{cases}0, & \text{if }x=0, \\ \omega x, & \text{if }x\not=0\text{ and }x\notin C_{d-1}, \\ \omega^ox, & \text{if }x\in C_{d-1}.\end{cases}
\]
Observe that $f$ is not additive (i.e., $f$ is not an endomorphism of the additive group $\IF_q$). Otherwise, since the function $g:\IF_q\rightarrow\IF_q$ with $g(x)=\omega x$ for all $x\in\IF_q$ is additive, we would have that $f-g$ is additive as well. But
\begin{align*}
|\ker(f-g)| &=|\{x\in\IF_q: (f-g)(x)=0\}|=|\{0\}\cup\bigcup_{i=0}^{d-2}{C_i}|=1+(d-1)\frac{q-1}{d} \\
&=1+d\cdot\frac{q-1}{d}-\frac{q-1}{d}=q-\frac{q-1}{d}\geq q-\frac{q-1}{2}>\frac{1}{2}q.
\end{align*}
Since $\ker(f-g)$ is a subgroup of the additive group $\IF_q$ and the order of a subgroup of a finite group $G$ divides $|G|$, we conclude that $\ker(f-g)=\IF_q$, i.e., that $f=g$, which is not the case, and we obtain the desired contradiction.

We claim that $f$ is a permutation of $\IF_q$. Since $f$ is a first-order cyclotomic mapping, this is equivalent to stating that $f$ permutes the cosets of $C$. Now, if $i=0,1,\ldots,d-2$, then $f(C_i)=\omega C_i=\omega(\omega^iC)=\omega^{i+1}C=C_{i+1}$, so it boils down to showing that $f(C_{d-1})=C_0$. But $f(C_{d-1})=\omega^oC_{d-1}$, so we need to show that $o\equiv1\Mod{d}$. By definition of $o$, we have
\[
\gcd(q-1,o+d-1)=\gcd(q-1,u\cdot d)=d.
\]
In particular, $d\mid o+d-1$, whence $d\mid o-1$ or, equivalently, $o\equiv1\Mod{d}$, as required. This concludes our argument that $f$ is a permutation of $\IF_q$.

Next, we claim that $f$ permutes $\IF_q^{\ast}$ cyclically. Since $f$ permutes the $d$ cosets of $C$ in $\IF_q^{\ast}$ cyclically, this is equivalent to claiming that the iterate $f^d$ permutes each coset of $C$ cyclically. But
\[
f^d(x)=\omega^{d-1}\omega^ox=\omega^{o+d-1}x
\]
for all $x\in\IF_q$, so this is equivalent to claiming that $\ord(\omega^{o+d-1})=\frac{q-1}{d}$ which, in turn, is equivalent to $\gcd(o+d-1,q-1)=d$, which we showed above.

It remains to show that the primitive root $\omega$ can be chosen such that $g:=f+\id$ is a permutation of $\IF_q$ that permutes $\IF_q^{\ast}$ cyclically. Observe that $g$ is also a first-order cyclotomic mapping of $\IF_q$, of the following form:
\[
g(x)=\begin{cases}0, & \text{if }x=0, \\ (\omega+1)x, & \text{if }x\not=0\text{ and }x\notin C_{d-1}, \\ (\omega^o+1)x, & \text{if }x\in C_{d-1}.\end{cases}
\]
From this, we see that $g$ will show the same \enquote{rough} mapping behavior (on the cosets of $C$) as $f$ as long as
\begin{itemize}
\item $\log_{\omega}(\omega+1)\equiv1\Mod{d}$ and
\item $\log_{\omega}(\omega^o+1)\equiv1\Mod{d}$.
\end{itemize}
Assuming these two conditions hold, $g$ will permute the elements of $\IF_q^{\ast}$ cyclically as long as
\[
\ord((\omega+1)^{d-1}(\omega^o+1))=\frac{q-1}{d}.
\]
Recall that we originally chose $\omega$ as a primitive root of $\IF_q$, and this is crucial for the argument that $f$ permutes $\IF_q^{\ast}$ cyclically. Therefore, what we need to show is that $\IF_q$ admits an element $\xi$ with the following properties:
\begin{enumerate}
\item $\ord(\xi)=q-1$.
\item $\ord((\xi+1)^{d-1}(\xi^o+1))=\frac{q-1}{d}$.
\item $\log_{\omega}(\xi+1)\equiv1\Mod{d}$.
\item $\log_{\omega}(\xi^o+1)\equiv1\Mod{d}$.
\end{enumerate}
Assuming that $q$ is large enough, we will show that such an element $\xi$ exists by an adaptation of the character sum method discussed in Section \ref{sec3} (which, in turn, is a generalization of Carlitz's character sum method from \cite{Car56a}). First, we need a few more technical observations.

Note that $o$ is not divisible by $p$. Indeed, if $o=ud-(d-1)\equiv0\Mod{p}$, it follows that $u\equiv e(d-1)\Mod{p}$, and $u$ was chosen specifically such that this congruence does \emph{not} hold. In particular, the polynomial $P:=T^o+1\in\IF_q[T]$ is squarefree, because
\[
\gcd(P,P')=\gcd(T^o+1,oT^{o-1})=1.
\]
Moreover, since $o$ is odd, we can write $T^o+1=(T+1)\cdot Q$ where
\[
Q:=T^{o-1}-T^{o-2}+T^{o-3}\mp\cdots-T+1.
\]
Since
\[
Q(-1)=o\not=0,
\]
we have $\gcd(Q,T+1)=1$. Furthermore, $Q(0)=1\not=0$, so $\gcd(Q,T)=1$ as well. In summary, we note that the three polynomials $T,T+1,Q\in\IF_q[T]$ are non-constant, squarefree and pairwise coprime (as in the assumptions of the Davenport-Weil Lemma \ref{davenportLem}).

Using Lemmas \ref{carlitzLem}(1) and \ref{discLogLem}, we see that the number of $\xi\in\IF_q$ that satisfy conditions (1)--(4) from above is
\[
N:=\sum_{\xi\in\IF_q}{f_{q-1}(\xi)f_{\frac{q-1}{d}}((\xi+1)^{d-1}(\xi^o+1))g_{d,1}(\xi+1)g_{d,1}(\xi^o+1)}.
\]
Using Lemma \ref{carlitzLem}(2) and the definition of $g_{d,i}$ from Lemma \ref{discLogLem}, we can rewrite this into
\begin{equation}\label{nRewrittenEq}
N=\frac{\phi(q-1)\phi(\frac{q-1}{d})}{(q-1)^2d^2}\sum_{z_1,z_2\mid q-1}{\frac{\mu(z_1)\mu(\frac{z_2}{\gcd(z_2,d)})}{\phi(z_1)\phi(\frac{z_2}{\gcd(z_2,d)})}\sum_{\psi_1,\psi_2: \atop \ord(\psi_i)\mid d}{\frac{1}{(\psi_1\psi_2)(\omega)}\sum_{\chi_1,\chi_2: \atop \ord(\chi_i)=z_i}{S(\vec{P},\vec{\nu})}}}
\end{equation}
where $\vec{P}=(P_1,P_2,P_3):=(T,T+1,Q)$, $\vec{\nu}=(\nu_1,\nu_2,\nu_3):=(\chi_1,\chi_2^d\psi_1\psi_2,\chi_2\psi_2)$ and, as by the definition in Lemma \ref{davenportLem},
\[
S(\vec{P},\vec{\nu})=\sum_{\xi\in\IF_q}{(\chi_1(\xi)\cdot(\chi_2^d\psi_1\psi_2)(\xi+1)\cdot(\chi_2\psi_2)(Q(\xi)))}.
\]
We now analyze this character sum using Lemma \ref{davenportPlusLem}. We write
\[
N=N_{\mathrm{large}}+N_{\mathrm{small}}
\]
where $N_{\mathrm{large}}$ is the sum of all the summands in formula (\ref{nRewrittenEq}) for which $S(\vec{P},\vec{\nu})$ is large, i.e., where $S(\vec{P},\vec{\nu})$ is of the form
\[
q-|\Null(P_1P_2P_3)|=q-|\Null(T(T^o+1))|>q-q^{\epsilon},
\]
where the last inequality uses that
\[
|\Null(T(T^o+1))|\leq\deg(T(T^o+1))=o+1=ud-(d-1)+1\leq ud<q^{\epsilon}.
\]
By Lemma \ref{davenportPlusLem}, $S(\vec{P},\vec{\nu})$ is of this form if and only if each of the three characters $\nu_i$ is the principal character $\chi_0$, which is equivalent to the three character equalities $\chi_1=\chi_0$, $\psi_2=\chi_2^{-1}$ and $\psi_1=\chi_2^{1-d}$. In particular, we then necessarily have
\[
z_1=\ord(\chi_1)=1
\]
and
\[
z_2=\ord(\chi_2)=\ord(\psi_2)\mid d
\]
as well as
\[
\ord(\psi_1)=\ord(\chi_2^{1-d})=\frac{z_2}{\gcd(z_2,d-1)}=z_2.
\]
Conversely, if $z$ is a divisor of $d$ and $\psi_1$ and $\psi_2$ are fixed multiplicative characters of $\IF_q$ of order $z$, then there are unique choices for the multiplicative characters $\chi_1$ and $\chi_2$ such that $\ord(\chi_2)=z$ and each character product $\nu_i=\nu_i(\chi_1,\chi_2,\psi_1,\psi_2)$ is principal.

Therefore, we see that the sum of the summands in formula (\ref{nRewrittenEq}) in which $S(\vec{P},\vec{\nu})$ is large equals
\begin{align*}
N_{\mathrm{large}}=&\frac{\phi(q-1)\phi(\frac{q-1}{d})}{(q-1)^2d^2}\sum_{z\mid d}{\frac{\mu(1)^2}{\phi(1)^2}\sum_{\psi_1,\psi_2:\ord(\psi_i)=z}{\frac{1}{(\psi_1\psi_2)(\omega)}}\left(q-|\Null(T(T^o+1))|\right)} \\
&=\frac{\phi(q-1)\phi(\frac{q-1}{d})}{(q-1)^2d^2}\left(q-|\Null(T(T^o+1))|\right)\sum_{z\mid d}{\left(\sum_{\ord(\psi)=z}{\psi(\omega)^{-1}}\right)^2} \\
&>\frac{\phi(q-1)\phi(\frac{q-1}{d})}{(q-1)^2d^2}\left(q-q^{\epsilon}\right)\sum_{z\mid d}{\mu(z)^2} \\
&=\frac{2\phi(q-1)\phi(\frac{q-1}{d})}{(q-1)^2d^2}\left(q-q^{\epsilon}\right),
\end{align*}
where the last equality uses our assumption that $d$ is a prime. Using that for each $\delta>0$, the Euler totient function $\phi$ grows faster than $n\mapsto n^{1-\delta}$, we conclude that if $q$ is large enough (relative to $d$), then
\[
N_{\mathrm{large}}\geq\frac{2(q-1)^{1-\delta}\left(\frac{q-1}{d}\right)^{1-\delta}}{(q-1)^2d^2}\cdot\left(q-q^{\epsilon}\right)=\frac{2}{d^{3-\delta}}\cdot\frac{q-q^{\epsilon}}{(q-1)^{2\delta}}\geq\frac{2}{d^3}\cdot\frac{1}{2}\cdot\frac{q}{q^{2\delta}}=\frac{1}{d^3}q^{1-2\delta}.
\]
It remains to show that asymptotically, $N_{\mathrm{small}}$ grows more slowly than $N_{\mathrm{large}}$. Using Lemma \ref{davenportPlusLem}(1), the bound $\deg(T(T^o+1))<q^{\epsilon}$ from above, as well as the triangle inequality (and observing that the absolute value of a multiplicative character value is always at most $1$), we find that
\begin{align*}
|N_{\mathrm{small}}|&\leq\frac{1}{d^2}\sum_{z_1,z_2\mid q-1}{\frac{1}{\phi(z_1)\phi(\frac{z_2}{\gcd(z_2,d)})}\sum_{\psi_1,\psi_2:\ord(\psi_i)\mid d}{\frac{1}{1}\sum_{\chi_1,\chi_2:\ord(\chi_i)=z_i}{\left(q^{\epsilon+\frac{1}{2}}\right)}}} \\
&=\frac{q^{\frac{1}{2}+\epsilon}}{d^2}\cdot\sum_{z_1,z_2\mid q-1}{\frac{d^2\phi(z_1)\phi(z_2)}{\phi(z_1)\phi(\frac{z_2}{\gcd(z_2,d)})}}=q^{\frac{1}{2}+\epsilon}\sum_{z_1,z_2\mid q-1}{\frac{\phi(z_2)}{\phi(\frac{z_2}{\gcd(z_2,d)})}} \\
&\leq q^{\frac{1}{2}+\epsilon}\sum_{z_1,z_2\mid q-1}{d}=dq^{\frac{1}{2}+\epsilon}\tau(q-1)^2.
\end{align*}
The number of divisors function $\tau$ grows more slowly than $n\mapsto n^{\delta}$. Therefore, if $q$ is large enough relative to $d$, we have
\[
|N_{\mathrm{small}}|\leq dq^{\frac{1}{2}+\epsilon}q^{2\delta}=dq^{\frac{1}{2}+\epsilon+2\delta}.
\]
With $\epsilon:=\frac{1}{8}$ and $\delta:=\frac{1}{16}$, this indeed grows more slowly than the lower bound $\frac{1}{d^3}q^{1-2\delta}$ for $N_{\mathrm{large}}$. Since $N_{\mathrm{large}}\to\infty$ as $q\to\infty$, this shows in particular that $N=N_{\mathrm{large}}+N_{\mathrm{small}}\geq 1$ if $q$ is large enough (relative to $d$), as we needed to show.

\section{Concluding remarks}\label{sec7}

In this section, we discuss some open questions relating to our Theorem \ref{mainTheo3}. The first is the following:

\begin{question}\label{openQues1}
Is it true that all large enough finite fields $\IF_q$ admit a non-additive complete mapping $f$ with $f(0)=0$ such that both $f$ and $f+\id$ permute $\IF_q^{\ast}$ cyclically?
\end{question}

Note that Theorem \ref{mainTheo3} states that this is true as long as we assume $d\mid q-1$ for some fixed integer $d>1$, i.e., as long as $q-1$ has some small nontrivial divisor $d$. However, this is not the case for every $q$ (think of $q-1$ being a Mersenne prime). For prime powers $q$ where $q-1$ has only large prime divisors, our method of using cyclotomic mappings breaks down, and one would need a different method for constructing $f$.

Both other open questions pertain to other combinations of cycle structures of $f$ and $f+\id$. For example, what about other cycle structures of index $d$ first-order cyclotomic permutations with only long cycles on $\IF_q^{\ast}$:

\begin{question}\label{openQues2}
Let $d$ be a positive integer, and let $\epsilon>0$. Is it true that there is a constant $C_3=C_3(d,\epsilon)$ such that if
\begin{itemize}
\item $q$ is a prime power with $q\geq C_3(d,\epsilon)$ and
\item $\gamma_1,\gamma_2$ are cycle types of index $d$ first-order cyclotomic permutations of $\IF_q$ all of whose cycles on $\IF_q^{\ast}$ are of length at least $\epsilon q$,
\end{itemize}
then there is an index $d$ first-order cyclotomic permutation $f$ of $\IF_q$ such that $f$ is a complete mapping, $\CT(f)=\gamma_1$ and $\CT(f+\id)=\gamma_2$?
\end{question}

Of course, further generalizations of Question \ref{openQues2} are also imaginable, such as asking for control over the cycle types of $f$ and $f+c\id$ for some $c\in\IF_q^{\ast}$, or perhaps even over the cycle types of $f+c_1\id,f+c_2\id,\ldots,f+c_r\id$ where $c_1,\ldots,c_r\in\IF_q$ are pairwise distinct.

Finally, we would like to discuss a variant of the situation in Theorem \ref{mainTheo3} already hinted at at the end of Subsection \ref{subsec1P1}. For the purposes of the remaining discussion, let us introduce the following non-standard terminology:

\begin{definition}\label{niceDef}
A complete mapping $f$ of a finite field $\IF_q$ is called \emph{special} if $f$ permutes the elements of $\IF_q$ cyclically, and $f+\id$ permutes $q-1$ elements of $\IF_q$ cyclically (and has one fixed point).
\end{definition}

We note that special complete mappings $f$ are the most extreme as far as simultaneous long cycles of $f$ and $f+\id$ are concerned. Indeed, it is not possible to have a complete mapping $f$ of $\IF_q$ such that both $f$ and $f+\id$ consist of just one $q$-cycle each. This is because $f+\id$ is a so-called \emph{orthomorphism} of $\IF_q$ (a permutation $g$ of $\IF_q$ such that $g-\id$ is also a permutation of $\IF_q$), and it is known (see e.g.~\cite[Lemma 4]{HK09a}) that every orthomorphism of $\IF_q$ has precisely one fixed point.

For the same reason, $\IF_q$ cannot have a special complete mapping if $q$ is even: In characteristic $2$, complete mappings and orthomorphisms are the same, and so if $f$ is a complete mapping of a finite field of characteristic $2$, then both $f$ and $f+\id$ have precisely one fixed point.

On the other hand, there are examples of odd-characteristic finite fields with a special complete mapping. The following is the only systematic (and possibly infinite, see below) family of examples known to the authors:

\begin{proposition}\label{niceProp}
Let $p$ be a prime such that $2$ is a primitive root modulo $p$. Then for every $b\in\IF_p^{\ast}$, the function
\[
f:\IF_p\rightarrow\IF_p,x\mapsto x+b,
\]
is a special complete mapping of $\IF_p$.
\end{proposition}

\begin{proof}
It is clear that $f$, being a nontrivial element of the regular representation of the additive group $\IZ/p\IZ$ on itself, is a $p$-cycle. Moreover, $p$ is necessarily odd by assumption, so $f+\id:x\mapsto 2x+b$ is also a permutation of $\IF_p$. Since $f+\id$ is an orthomorphism of $\IF_p$, it has precisely one fixed point (namely $-b$). Moreover, observe that for each positive integer $\ell$, we have
\[
(f+\id)^{\ell}(x)=2^{\ell}x+b(1+2+2^2+\cdots+2^{\ell-1})=2^{\ell}x+(2^{\ell}-1)b.
\]
By assumption, if $\ell\in\{2,3,\ldots,p-2\}$, then $2^{\ell}\not\equiv1\Mod{p}$, so $2^{\ell}-1$ is nonzero in $\IF_p$ and thus a unit in $\IF_p$. It follows that the equation
\[
2^{\ell}x+(2^{\ell}-1)b=(f+\id)^{\ell}(x)=x
\]
has precisely one solution (the fixed point $-b$ of $f+\id$), so $f$ has no cycles of any length $\ell\in\{2,3,\ldots,p-2\}$. This means that $f+\id$ must consist of its unique fixed point $-b$ and a $(p-1)$-cycle, as required.
\end{proof}

It is not known unconditionally whether there are infinitely many primes that have $2$ as a primitive root. However, if the Generalized Riemann Hypothesis holds, then there are infinitely many such primes; in fact, one then has that the density of such primes is equal to the so-called Artin constant
\[
C_{\mathrm{Artin}}=\prod_{\text{primes }p}{\frac{1}{p(p-1)}}=0.3739\ldots,
\]
see \cite[Theorem in Section 7]{Hoo67a}.

It would be interesting to gain a better understanding of special complete mappings of finite fields. In this spirit, using GAP \cite{GAP4}, the authors counted the precise number of special complete mappings of $\IF_q$ for all odd prime powers $q\leq13$ with a brute force algorithm, and they verified that $\IF_q$ with $q\in\{17,19,23,25\}$ has at least one special complete mapping (using a random search algorithm for $q=17$ and $q=23$). The results are summarized in Table \ref{table1}. For $\IF_{25}$, we denote by $\omega$ any fixed root (in $\IF_{25}$) of the Conway polynomial $T^2-T+2\in\IF_5[T]$.

\begin{table}[h]\centering
\begin{tabular}{|c|c|c|}\hline
$q$ & $N_q$ & exemplary special complete mapping of $\IF_q$ \\ \hline
$3$ & $2$ & $x\mapsto x+1$ \\ \hline
$5$ & $4$ & $x\mapsto x+1$ \\ \hline
$7$ & $36$ & $(0,6,4,1,3,5,2)$ \\ \hline
$9$ & $0$ & there is none \\ \hline
$11$ & $760$ & $x\mapsto x+1$ \\ \hline
$13$ & $22212$ & $x\mapsto x+1$ \\ \hline
$17$ & unknown & $(0,3,9,10,13,5,11,14,12,2,8,16,1,4,7,6,15)$ \\ \hline
$19$ & unknown & $x\mapsto x+1$ \\ \hline
$23$ & unknown & $(0,20,7,14,18,5,6,11,8,2,10,15,9,13,16,21,17,22,19,12,1,4,3)$ \\ \hline
$25$ & unknown & \thead{$(0,4\omega+1,4\omega+4,\omega,\omega+4,\omega+1,4,\omega+3,2\omega+1,3\omega,\omega+2,4\omega+3,1,$ \\ $2,3\omega+1,3\omega+2,3\omega+4,4\omega+2,4\omega,2\omega+4,3\omega+3,3,2\omega,2\omega+3,2\omega+2)$} \\ \hline
\end{tabular}
\caption{Numbers $N_q$ and examples of special complete mappings of small finite fields $\IF_q$.}
\label{table1}
\end{table}

These computational investigations motivate the following open question:

\begin{question}\label{openQues3}
Can the prime powers $q$ such that $\IF_q$ admits a special complete mapping be characterized? Do all large enough odd prime powers have this property?
\end{question}

\end{document}